\documentclass[a4paper]{article}

\usepackage{amsmath,amsthm,amssymb}
\usepackage[top=3cm,bottom=3cm,left=3cm,right=3cm]{geometry} 
\usepackage[applemac]{inputenc}
\usepackage[all]{xy} 
\usepackage[T1]{fontenc} 
\usepackage{textcomp} 
\usepackage{slashed}
\usepackage{color}
\usepackage{mathrsfs}
\usepackage{leftidx}
\usepackage{bm}
\usepackage{url}
\usepackage{hyperref}

\newcommand{\TT}{\mathbf{T}}
\DeclareMathOperator{\im}{im}
\DeclareMathOperator{\re}{Re}

\DeclareMathOperator{\tr}{Tr}

\DeclareMathOperator{\diam}{diam}

\DeclareMathOperator{\Rm}{Rm}
\DeclareMathOperator{\Ric}{Ric}
\DeclareMathOperator{\RRic}{\mathbf{Ric}}
\DeclareMathOperator{\RRm}{\mathbf{Rm}}
\DeclareMathOperator{\RR}{\mathbf{R}}

\DeclareMathOperator{\SO}{SO}

\DeclareMathOperator{\SU}{SU}
\DeclareMathOperator{\GL}{GL}

\DeclareMathOperator{\Vol}{Vol}

\newcommand{\dd}{\mathrm{d}}

\newcommand{\R}{\mathbb R}

\newcommand{\Z}{\mathbb Z}
\newcommand{\N}{\mathbb N}

\newcommand{\diff}{\mathrm{d}}
\newcommand{\del}{\partial}

\newcommand{\dvol}{\mathrm{dvol}}

\renewcommand{\tilde}{\widetilde}

\theoremstyle{plain}
	\newtheorem{theorem}{Theorem}
	\newtheorem{proposition}[theorem]{Proposition}
	\newtheorem{lemma}[theorem]{Lemma}
	\newtheorem{corollary}[theorem]{Corollary}
	\newtheorem{conjecture}[theorem]{Conjecture}
	\newtheorem{question}[theorem]{Question}
\theoremstyle{definition}
	\newtheorem{definition}[theorem]{Definition}
	\newtheorem{remark}[theorem]{Remark}

\theoremstyle{plain}
	\newtheorem*{theorem*}{Theorem}
	\newtheorem*{proposition*}{Proposition}
	\newtheorem*{lemma*}{Lemma}
	\newtheorem*{corollary*}{Corollary}
	\newtheorem*{conjecture*}{Conjecture}
\theoremstyle{definition}
	\newtheorem*{definition*}{Definition}
	\newtheorem*{remark*}{Remark}
	\newtheorem*{remarks*}{Remarks}

\usepackage{authblk}
\title{A report on the hypersymplectic flow}
\author[1]{Joel Fine}
\affil[1]{Département de mathématiques, Université libre de Bruxelles.}
\author[2]{Chengjian Yao}
\affil[2]{Institute of Mathematical Sciences, ShanghaiTech University.}
\date{}                     

\numberwithin{equation}{section}
\numberwithin{theorem}{section}

\begin{document}

\maketitle

\begin{center}
Written for Sir Simon Donaldson, on the occasion of his 60th birthday.
\vspace{\baselineskip}

\end{center}

\begin{abstract}
This article discusses a relatively new geometric flow, called the \emph{hypersymplectic flow}. In the first half of the article we explain the original motivating ideas for the flow, coming from both 4-dimensional symplectic topology and 7-dimensional $G_2$-geometry. We also survey recent progress on the flow, most notably an extension theorem assuming a bound on scalar curvature. The second half contains new results. We prove that a complete torsion-free hypersymplectic structure must be hyperk\"ahler. We show that a certain integral bound involving scalar curvature rules out a finite time singularity in the hypersymplectic flow. We show that if the initial hypersymplectic structure is sufficiently close to being point-wise orthogonal then the flow exists for all time. Finally, we prove convergence of the flow under some strong assumptions including, amongst other things, long time existence.
\end{abstract}


\section{Summary of contents}

The first aim of this article is to popularise a relatively new geometric flow, called the hypersymplectic flow. The second aim is to explain some new results about extension and convergence of this flow. In \S\ref{motivation} we give the background motivation for hypersymplectic stuctures, coming from 4-dimensional symplectic topology and 7-dimensional $G_2$-geometry. In \S\ref{hypersymplectic-flow-to-date} we summarise what is already known about the hypersymplectic flow. We do this partly for context but also because several of the equations and arguments discussed here will be reused later. The new results are proved in \S\ref{new-results}. They include the following:
\begin{itemize}
\item
Theorem~\ref{torsion free-hyperkahler}: a  torsion-free complete hypersymplectic structure is hyperk\"ahler.
\item
Theorem~\ref{extendible}: an extension result for the hypersymplectic flow under an integral bound on the torsion.
\item
Theorem~\ref{C0-long-time-existence}: if the initial hypersymplectic structure is sufficiently close (in $C^0$) to being pointwise orthogonal then the hypersymplectic flow exists for all time. 
\item
Theorems~\ref{first-convergence} and~\ref{second-convergence}: assuming the hypersymplectic flow exists for all time, then certain geometric bounds imply the flow converges to a hyperk\"ahler structure. 
\end{itemize}

\subsection*{Acknowledgements}

JF was supported by ERC consolidator grant 646649 ``SymplecticEinstein''.
\section{Motivation}\label{motivation}

\subsection{Hypersymplectic structures} 

One starting point for the study of hypersymplectic structures comes from the following folklore conjecture in 4-dimensional symplectic topology.

\begin{conjecture}
\label{folklore-conjecture}
Let $M$ be a compact 4-manifold and $\omega$ a symplectic form on $M$ with $c_1(M,\omega) = 0$. If  $\pi_1(M) = 0$ or $b_+(M)=3$ then there exists an integrable almost complex structure $J$ on $M$ which is compatible with $\omega$, making $(M,J,\omega)$ a K\"ahler surface.
\end{conjecture}

We make the following remarks:
\begin{enumerate}
\item 
The conditions $\pi_1(M)=0$ or $b_+(M)=3$ are necessary. The Kodaira--Thurston manifold \cite{kodaira,thurston} is a homogenous compact symplectic 4-manifold, with $c_1=0$; indeed the tangent bundle is symplectically trivial. It has no K\"ahler metric, since $b_1=3$ whereas K\"ahler manifolds have $b_1$ even. It has $\pi_1 \neq 0$ and $b_+ = 2$. 
\item
Conjecture~\ref{folklore-conjecture}, if true, describes a purely 4-dimensional phenomenon. There are infinitely many simply-connected compact symplectic 6-manifolds with $c_1=0$ yet which admit no compatible complex structure \cite{fine-panov2,fine-panov3,Ak}.
\item 
The only compact K\"ahler surfaces with $c_1=0$ are complex tori or K3 surfaces. So one consequence of the conjecture is that $M$ is either diffeomorphic to $\mathbb{T}^4$ or the real manifold underlying K3 surfaces.

The best that is known in this direction is that under the hypotheses of Conjecture~\ref{folklore-conjecture}, if $\pi_1(M) = 0$ then $M$ is \emph{homeomorphic} to  a K3 surface. This follows from work of Morgan--Szab\'o, Bauer and Li to determine the integral homology of $M$ \cite{MSa, Ba,Li}, and then the celebrated work of Freedmann \cite{freedman} to determine the homeomorphism type. 

Of course, the jump from homeomorphism to diffeomorphism is large in dimension $4$. Moreover, even if we assume $M$ is diffeomorphic to a K3 surface or $\mathbb{T}^4$ it is still unknown whether or not $\omega$ can be made into a K\"ahler metric. 
\end{enumerate}

Conjecture~\ref{folklore-conjecture} seems currently out of reach. To gain a foothold, Donaldson has formulated a simpler conjecture, which may turn out to be more tractable. To state it we first need the definition of a hypersymplectic structure. 

\begin{definition}
Let $M$ be a 4-manifold. A triple $\underline{\omega}=(\omega_1,\omega_2,\omega_3)$ of symplectic forms is called a \emph{hypersymplectic structure} on $M$ if any non-zero linear combination $a_1 \omega_1 +a_2 \omega_2 + a_3 \omega_3$ of the forms is again a symplectic form (where $(a_1, a_2, a_3 )\in \R^3\setminus\{0\}$).
\end{definition}

The obvious example, and reason for the name, is the triple of K\"ahler forms associated to a hyperk\"ahler metric. Donaldson's conjecture is that, up to isotopy and on compact manifolds, these are essentially the \emph{only} examples. 

\begin{conjecture}[Donaldson \cite{D09}]
\label{skd-conjecture}
Let $\underline{\omega}$ be a hypersymplectic structure on a compact 4-manifold $M$. Then there is an isotopy $\underline{\omega}(t)$ of cohomologous hypersymplectic structures for $0 \leqslant t \leqslant1$, taking $\underline{\omega} = \underline{\omega}(0)$ to a triple~$\underline{\omega}(1)$ that is hyperk\"ahler. I.e., there is a hyperk\"ahler metric on $M$ for which the family of K\"ahler forms are generated by the components of $\underline{\omega}(1)$.
\end{conjecture}

Donaldson's Conjecture is a special case of Conjecture~\ref{folklore-conjecture}. First note that the hypotheses of Conjecture~\ref{skd-conjecture} imply those of Conjecture~\ref{folklore-conjecture}.

\begin{lemma}
If $\underline{\omega}$ is a hypersymplectic structure on $M$ then $c_1(M,\omega_1) =0$. Moreover, if $M$ is compact then $b_+(M)=3$.
\end{lemma}

\begin{proof}[Sketch of proof]
We begin with $c_1(M,\omega_1) =0$. Let $J$ be an almost complex structure on $M$ compatible with $\omega_1$ and $g(\cdot, \cdot) = \omega_1(J \cdot, \cdot)$ the corresponding Riemannian metric. The bundle of self-dual 2-forms of $g$ has the form $\Lambda^+ =\left\langle \omega_1 \right\rangle \oplus K_J$ where $K_J \subset \Lambda^2$ is the bundle of 2-forms which are the real parts of $(2,0)$-forms with respect to $J$. Conversely, given any rank~2 subbundle $K \subset \Lambda^2$ such that the wedge product is positive definite on $\left\langle  \omega_1\right\rangle\oplus K$ then there is an almost complex structure $J$, compatible with $\omega_1$, such that $K=K_J$. 

Let $K = \{ \chi \in \left\langle  \omega_1, \omega_2, \omega_3 \right\rangle : \omega_1 \wedge \chi = 0\}$. Now $\left\langle \omega_1 \right\rangle\oplus K = \left\langle \omega_1, \omega_2, \omega_3 \right\rangle$; by the definition of a hypersymplectic structure, the wedge product is positive definite here. It follows that $K =K_J$ for some almost complex $J$, compatible with $\omega_1$. It is now a simple matter to write down a nowhere vanishing section of $K_J$, showing that $c_1(M,J) = 0$. 

To see that when $M$ is compact  $b_+(M)=3$, first note that the $J$ described above gives a metric $g$ for which $\omega_1, \omega_2, \omega_3$ are all closed self-dual 2-forms. This shows that $b_+(M) \geqslant3$. Meanwhile a result of Bauer \cite{Ba} shows that for a symplectic 4-manifold with $c_1=0$, $b_+(M) \leqslant 3$. (Note that Bauer's Theorem is relatively deep, relying on Seiberg--Witten theory. It would be interesting to know if there was a more direct proof that the existence of a hypersymplectic structure forces $b_+(M)=3$.)
\end{proof}

Next we explain why Donaldson's Conjecture implies Conjecture~\ref{folklore-conjecture} in the case of hypersymplectic structures. Let $(M, \underline{\omega})$ be a compact hypersymplectic 4-manifold and let $\underline{\omega}(t)$ be an isotopy of cohomologous hypersymplectic structures starting at this given one and ending at a hyperk\"ahler triple $\underline{\omega}(1)$. There is certainly an integrable complex structure $J$ compatible with $\omega_1(1)$. Meanwhile Moser's argument applied to the path of cohomologous symplectic forms $\omega_1(t)$ gives the existence of a diffeomorphism $f$ with $f^* \omega_1(1) = \omega_1(0)$. It follows that $(M, f^*J, \omega_1(0))$ a K\"ahler surface. 
\vspace{\baselineskip}

We now turn to a geometric flow of hypersymplectic structures, called the \emph{hypersymplectic flow}, which we hope will ultimately lead to an isotopy $\underline{\omega}(t)$ proving Donaldson's Conjecture. 

\subsection{A brief primer on $G_2$-structures}

The hypersymplectic flow has its origins in $G_2$ geometry. In this section we quickly review the required parts of the study of $G_2$-structures. For more details and justifications, see \cite{Bryant,LW1}.

Let $X$ be an oriented 7-manifold and $\phi \in \Omega^3(X)$ a 3-form. Using $\phi$ we can define a symmetric bilinear form $\beta_\phi$ on $X$ with values in $\Lambda^7T^*X$ as follows:
\begin{equation}
\beta_\phi(u,v) = \frac{1}{6}\iota_u \phi \wedge \iota_v \phi \wedge \phi
\end{equation}
\begin{definition}
The 3-form $\phi$ is called a \emph{$G_2$-structure on $X$} if $\beta_\phi$ is positive definite. More precisely, if $\nu$ is any choice of positive nowhere vanishing 7-form on $X$, then $\phi$ is a $G_2$-structure if $\beta_\phi/\nu$ defines a Riemannian metric on $X$.
\end{definition}

When $\phi$ is a $G_2$-structure, there is in fact a canonical choice of positive 7-form on $X$. To explain this, note that given any nowhere vanishing 7-form $\nu$, we obtain a metric $g_\nu = \beta_\phi/\nu$. This metric has, in turn, a metric volume form (a positive 7-form of unit length with respect to $g_\nu$) which we denote by $\dvol(\nu)$ to indicate its dependence on $\nu$. By considering the effect of scaling $\nu$, one sees that there is a unique choice of $\nu$ for which $\dvol(\nu) = \nu$. In this way a $G_2$-structure induces a distinguished 7-form and hence metric on $X$. We write these structures $\nu_\phi$ and $g_\phi$ respectively. 

The prototype of a $G_2$-structure comes from the octonians. Identify $\R^7$ with the vector space of imaginary octonians and define a 3-form $\phi_0$ on $\R^7$ by $\phi_0(u,v,w) = \left\langle \im(u \times v),w \right\rangle$. Here $u \times v$ is the octonian product of $u$ and $v$, whilst $\left\langle \cdot, \cdot \right\rangle$ denotes the Euclidean inner product. The 3-form $\phi_0$ is the 7-dimensional analogue of the vector triple product in $\R^3$. The resulting metric $g_{\phi_0}$ on $\R^7$ is just the standard Euclidean metric we started with. It turns out that, up to the action of $\GL^+(7,\R)$, $\phi_0$ is the unique linear $G_2$-structure on $\R^7$. The stabiliser of $\phi_0$ is isomorphic to the exceptional Lie group $G_2$, which in this way arises as a subgroup $G_2 \subset \SO(7)$. It is from here that $G_2$-structures get their name. 

The interest in $G_2$-structures comes from the search for 7-dimensional Riemannian manifolds with holonomy $G_2$. 

\begin{definition}
A $G_2$-structure $\phi$ is said to be \emph{torsion free}	if $\nabla \phi = 0$. Here, $\nabla$ is the Levi-Civita connection of the metric $g_\phi$, determined by $\phi$. (In particular, this equation is non-linear in $\phi$.)

If $\phi$ is torsion free, then the holonomy along loops based at $x \in X$ must preserve $\phi(x)$. It follows that the holonomy group is isomorphic to a subgroup of $G_2$. (We remark in passing that the term ``torison free'' comes from the language of $G$-structures.)
\end{definition}

One of the reasons to be interested in metrics with $\nabla \phi = 0$ is that they are automatically \emph{Ricci flat}.  This is similar to Calabi--Yau metrics, i.e., K\"ahler metrics with a parallel holomorphic volume form, which have holonomy $\SU(n)$. Again the constant differential form forces the Ricci curvature to vanish. The analogy is particularly close in real dimension~6: if $Z$ is a Calabi--Yau threefold with K\"ahler form $\omega$ and holomorphic volume form $\Omega$ then $Z \times S^1$ carries a torsion free $G_2$-structure, $\phi = \omega \wedge \diff \theta + \re \Omega$. In terms of holonomy, this corresponds to the subgroup $\SU(3) \subset G_2$ of elements which fix a given direction in $\R^7$.

It turns out that the equation $\nabla \phi = 0$ is equivalent to a system of equations which is at first sight less restrictive.

\begin{proposition} (See, for example, \cite{Bryant,LW1}.)
A $G_2$-structure $\phi$ is torsion free if and only if 
\begin{equation}
\diff \phi = 0\quad \text{ and } \quad \diff^*\phi = 0
\end{equation}
(Here the codifferential $\diff^*$ is defined via the metric $g_\phi$ determined by $\phi$ and so, again, this equation is non-linear in $\phi$.)
\end{proposition}

In one direction this is obvious, the real strength of the Proposition is that $\phi$ being closed and coclosed is enough to force $\nabla \phi = 0$. Again one can draw an analogy with K\"ahler geometry: if the associated $(1,1)$-form $\omega$ of a Hermitian metric is closed then (using integrability of the complex structure) it follows that $\nabla \omega = 0$. 

This description of torsion free $G_2$-structures leads to the following natural question.

\begin{question}\label{G2-question}
Let $\phi$ be a closed $G_2$-structure. Does there exist a cohomologous torsion free $G_2$-structure $\psi \in [\phi]$?
\end{question}

Pushing the analogy with Calabi--Yau threefolds still further, one might think of this as a little similar to the Calabi--Yau theorem: given a complex manifold with trivial canonical bundle and a K\"ahler class $\kappa$, Yau proved that there is a unique K\"ahler form $\omega \in \kappa$ for which the holomorphic volume forms are parallel \cite{Yau}. We should stress, however, that there are important differences between $G_2$ and Calabi--Yau geometries. In particular, there are closed $G_2$-structures on compact 7-manifolds which admit no torsion free $G_2$-structure \cite{Fer}, and so the strongest ``$G_2$ analogue'' of Yau's theorem is false.

\subsection{From hypersymplectic to $G_2$}

In \cite{Donaldson} Donaldson observes how to turn a hyperpsymplectic structure into a $G_2$-structure (and conversely, certain 3-dimensional families of hypersymplectic structures are one of the key pieces of data in the description of general closed $G_2$-structures with co-associative fibrations). Let $M$ be a 4-manifold with hypersymplectic structure $\underline{\omega} = (\omega_1, \omega_2, \omega_3)$ and consider the following 3-form $\phi$ on $M \times \mathbb{T}^3$, where $\mathbb{T}^3$ is the 3-torus with angular coordinates $(t^1,t^2,t^3)$:
\begin{equation}\label{toric-G2}
\phi = \diff t^{123} - \diff t^1 \wedge \omega_1 - \diff t^2 \wedge \omega_2 - \diff t^3 \wedge \omega_3
\end{equation}
The fact that $\underline{\omega}$ is hypersymplectic implies $\phi$ is a closed $G_2$-structure. (See, for example, Lemma~2.2 of \cite{FY}). 

The metric $g_\phi$ on $M \times \mathbb{T}^3$ has a purely 4-dimensional description. Firstly, we define a volume form $\mu_{\underline{\omega}}$ on $M$ as follows. Given any nowhere vanishing positive 4-form $\mu$ on $M$, consider the symmetric 3-by-3 matrix
\begin{equation}
Q(\mu)_{ij} = \frac{\omega_i \wedge \omega_j}{2\mu}
\label{Q-definition}
\end{equation}
We let $\mu_{\underline{\omega}}$ be the unique positive 4-form for which $\det Q = 1$. We can now define a Riemannian metric $g_{\underline{\omega}}$ on~$M$:
\begin{equation}
g_{\underline{\omega}}(u,v) = \frac{1}{6}
\frac{ \epsilon^{ijk} 
\iota_u \omega_i \wedge \iota_v \omega_j \wedge \omega_k}{\mu_{\underline{\omega}}}
\end{equation}
(Here and throughout we use the summation convention that repeated indices are summed over 1,2,3; moreover $\epsilon^{ijk}$ is the sign of the permutaion $(ijk)$). One can check that this symmetric bilinear form is indeed a Riemmanian metric, because $\underline{\omega}$ is hypersymplectic. Moreover, the forms $\omega_i$ are actually self-dual with respect to $g_{\underline{\omega}}$, and $Q_{ij} =\frac{1}{2}\langle\omega_i, \omega_j\rangle$ is (half of) the matrix of their  $g_{\underline{\omega}}$-inner-products. Finally, the 7-dimensional metric $g_\phi$ on $M \times \mathbb{T}^3$ is related to $g_{\underline{\omega}}$ and $Q$ by 
\begin{equation}
g_{\phi} =  g_{\underline{\omega}} + Q_{ij} \diff t^i \otimes \diff t^j
\end{equation}
(where we have abused notation in writing $g_{\underline{\omega}}$ for the pull-back of this metric from $M$ to $M \times \mathbb{T}^3$). The proofs of all of these assertions can be found in \cite{FY}.

To summarise, a hypersymplectic structure $\underline{\omega}$ determines two key pieces of data: a Riemannian metric $g_{\underline{\omega}}$, which makes $\underline{\omega}$ self-dual, and a matrix valued function $Q$, given by the inner-products of the components of $\underline{\omega}$, taking values in the space of positive definite matrices with determinant~1. If we interpret $Q$ as a family of (flat) metrics on $\mathbb{T}^3$ parametrised by $M$ then, together with $g_{\underline{\omega}}$, we obtain a metric on $M \times \mathbb{T}^3$ which is induced by the closed $G_2$-structure $\phi$ as defined above in \eqref{toric-G2}.

We are now in position to interpret Donaldson's Conjecture~\ref{skd-conjecture} in terms of Question~\ref{G2-question}. If $\underline{\omega}(t)$ is a path of cohomologous hypersymplectic structures ending at a hyperk\"ahler triple, then the corresponding $G_2$-structures $\phi(t)$ are cohomologous and end at a torsion free $G_2$-structure. Indeed, the holonomy of $M \times \mathbb{T}^3$ lies in $\SU(2) \subset G_2$, the subgroup of $G_2$ which fixes a copy of the imaginary quaternions inside the imaginary octonians. From this perspective, Donaldson's Conjecture asserts that, for those closed $G_2$-structures coming from hypersymplectic structures, you can always deform them to be torsion free. 

We close this subsection with a comment on notation. Given a hypersymplectic structure $\underline{\omega}$ we now have two different Riemannian manifolds in play, the 4-dimensional $(M,g_{\underline{\omega}})$ and the 7-dimensional $(M \times \mathbb{T}^3, g_{\phi})$. To help distinguish between them we will use bold characters for all 7-dimensional geometric quantities, for example writing $\bm{g} = g_{\phi}$ for the 7-dimensional metric and $\mathbf{Rm}$ for its curvature, whilst $g = g_{\underline{\omega}}$ and $\Rm$ denote the 4-dimensional metric and curvature.

\subsection{The $G_2$-Laplacian flow}\label{background}

The hypersymplectic flow is a special case of a flow defined for general $G_2$-structures, called the \emph{$G_2$-Laplacian flow}. This flow was introduced independently by Bryant and Hitchin \cite{Bryant,Hitchin} as a way to approach Question~\ref{G2-question}. A path of closed $G_2$-structures $\phi(t)$ evolves according to the $G_2$-Laplacian flow if 
\begin{equation}
\frac{\del \phi}{\del t} = \Delta \phi
\label{G2-flow}
\end{equation}
Here, $\Delta = \Delta_{g_{\phi(t)}}$ is the Hodge Laplacian on 3-forms determined by the metric, and hence $\phi$, at time $t$. Some remarks are in order. 
\begin{enumerate}
\item
A closed $G_2$-structure $\phi$ is a fixed point of the flow if and only if it is torsion free (This follows from the computation, equation~(6.14) in \cite{Bryant}, of the pointwise change in the volume form.) 
\item
Under the flow \eqref{G2-flow}, the cohomology class of $\phi(t)$ is constant. One might hope then that the flow will deform a given closed $G_2$-structure into a cohomologous one which is torsion free.
\item
The flow \eqref{G2-flow} is actually a gradient flow. This point of view is due to Hitchin \cite{Hitchin}. Given a closed $G_2$-structure $\phi$, write $\mathcal{V}(\phi)$ for the total volume of the associated metric. Write $\mathcal{H}$ for the space of all closed $G_2$-structures in a fixed cohomology class. Then, with a suitable Riemannian metric on $\mathcal{H}$, \eqref{G2-flow} is the gradient flow of $\mathcal{V}$. (Details of this calculation can be found in \cite{FY}.)
\end{enumerate}

The first step in studying \eqref{G2-flow} is to understand short time existence, which ultimately comes down to understanding the dependence of $\Delta$ on $\phi$. A simple observation is that this dependence means that \eqref{G2-flow} is not linear in $\phi$. More subtly, despite the fact that the Laplacian on the right-hand side is \emph{positive}, the complicated way in which $\Delta$ depends on $\phi$ means that \eqref{G2-flow} behaves like a forward heat flow, and not a backward one. The flow \eqref{G2-flow} is not genuinely parabolic, however, since it is invariant under diffeomorphisms. In \cite{BX}, Bryant and Xu showed how to gauge fix the flow to avoid this problem and hence prove short time existence. The idea is similar to DeTurck's proof of short-time existence of the Ricci flow, although the technicalities are more involved. (See also the discusion in \cite{LW1}.) 

\begin{theorem}[Bryant--Xu, \cite{BX}]
Let $X$ be a compact 7-manifold and $\phi_0$ a closed $G_2$-structure on $X$. Then there exists $\epsilon >0$ and $\phi(t)$ a path of closed $G_2$-structures for $t \in [0,\epsilon)$ solving the $G_2$-Laplacian flow \eqref{G2-flow} with $\phi(0) = \phi_0$. Moreover, the solution is unique for as long as it exists.
\end{theorem}
 
The $G_2$-Laplacian flow is, in fact, a close cousin of the Ricci flow. To see this, one computes the evolution of the metric $\bm{g}(t) = g_{\phi(t)}$ under the flow. As a notational convenience,  put 
\begin{equation}
\TT = - \frac{1}{2} \diff ^* \phi
\end{equation}
This 2-form is called \emph{the torsion} of $\phi$. With this in hand, we can give the evolution equation for $\bm{g}(t)$.
\begin{proposition}[Lotay--Wei \cite{LW1}, equation (3.6)]
If a closed $G_2$-structure $\phi(t)$ evolves according to the $G_2$-Laplacian flow~\eqref{G2-flow}, the corresponding Riemannian metric $\bm{g}(t) = g_{\phi(t)}$ obeys
\begin{equation}\label{metric-evolution-G2}
\frac{\del \bm{g}_{ij}}{\del t}
=
- 2 \RRic_{ij} - \frac{2}{3}|\TT|^2 g_{ij} - 4 \TT_{i}^{\phantom{i}k}\TT_{kj}
\end{equation}
\end{proposition}

The terms involving $\TT$ should be thought of as giving a lower order correction to Ricci flow, in the sense that the Ricci curvature is determined by $\bm\nabla \TT$, $\TT$ and $\phi$; see equation~(2.29) in \cite{LW1}. Inspired by this, much work has been done to extend results from Ricci flow to the $G_2$-Laplacian flow \cite{LW1, LW2, LW3, L}. One important result of this kind---and which will be important to us in what follows---is an extension criterion due to Lotay--Wei. Recall that if a finite time singularity occurs in the Ricci flow, then the curvature must blow up (due to Hamilton \cite{hamilton}). To state Lotay--Wei's analogous result for the $G_2$-Laplacian flow we first make a definition. Let $\Lambda$ denote the following quantity
\begin{equation}\label{Lambda}
\Lambda(\phi) = \sup_{X} \left( |\RRm(g_\phi)|^2 + |\bm\nabla \TT (\phi(t))|^2\right)^{1/2}
\end{equation}

\begin{theorem}[Lotay--Wei]\label{LW-extension}
Let $\phi(t)$ be a path of closed $G_2$-structures solving the $G_2$-Laplacian flow on a maximal time interval $[0,s)$ with $s < \infty$ on a compact manifold. Then
\[
\lim_{t \nearrow s} \Lambda(\phi(t)) = \infty
\]
\end{theorem}
%

An important problem is to find conditions under which the $G_2$-Laplacian flow exists for long time. One such result is due to Lotay--Wei, who show that torsion free $G_2$-structures are dynamically stable under $G_2$-Laplacian flow \cite{LW2}: if $\phi(0)$ is sufficiently close to a given cohomologous torsion free $G_2$-structure $\psi$ then the flow exists for all time and converges modulo diffeomorphisms to $\psi$.  Homogeneous examples of the $G_2$-Laplacian flow have been investigated by Lauret \cite{Lauret1} using the bracket/algebraic soliton approach giving, in particular, many examples of long time existence. In \cite{HWY} Huang--Wang--Yao studied a type of $G_2$-structure on the torus $\mathbb{T}^7$ with $\mathbb{T}^6$ symmetry; they show long time existence of the flow and convergence (modulo diffeomorphisms) to a standard flat structure. Fino--Raffero \cite{FR2} studied so-called ``extremely Ricci-pinched'' $G_2$-structures, a special class of $G_2$-structure introduced by Bryant \cite{Bryant} (together with the predating examples of Bryant \cite{Bryant} and Lauret \cite{Lauret1}), and showed the long time existence of the flow; in these cases the flow does \emph{not} converge as $t\to \infty$.  In \cite{LL}, Lambert--Lotay proved the long-time existence and convergence of the $G_2$-Laplacian flow in the case of semi-flat coassociative $\mathbb{T}^4$-fibrations, via the reduction of the $G_2$-Laplacian flow to a spacelike mean curvature flow in $H^2(\mathbb{T}^4)=\mathbb{R}^{3,3}$ (the connection with spacelike mean curvature flow is another observation due to Donaldson \cite{Donaldson}). 

It should be noted that among all these examples, there is no \emph{compact} example of finite time singularity. We should mention, however, an interesting example due to Lauret \cite{Lauret2} of a complete shrinking closed $G_2$-Laplacian soliton, which thus becomes singular in finite time. 
 
All of the known compact examples in which the flow exists for long time but does not converge at infinity have a common property: the volume along the flow increases without bound. In such a cohomology class, Hitchin's volume functional $\mathcal{V}$ is unbounded (see item 3 of section \ref{background}). For hypersymplectic structures, however, there is a uniform bound on the volume in terms of the cohomology classes of the $\underline\omega$. Since $\det Q =1$, we have that $\tr Q \geqslant3$ and so Equation~\eqref{Q-definition} implies that for any hypersymplectic stucture $\underline\omega$, the total volume satisfies
\begin{equation}
\int_{M \times \mathbb{T}^3} \nu_\phi = \int_M \mu_{\underline\omega} \leqslant\frac{1}{6} \int_M \omega_1^2 + \omega_2^2 + \omega_3^2
\label{volume-bound}
\end{equation}
where the right-hand side depends only on the cohomology classes of the $\omega_i$.


\section{The hypersymplectic flow}\label{hypersymplectic-flow-to-date}

\subsection{Evolution equations}

In \cite{FY}, the authors began the study of the $G_2$-Laplacian flow in the setting of hypersymplectic structures, with the (as yet unfulfilled!) aim of proving Donaldson's Conjecture. In this section we review this work, setting the scene for the new results which we will explain in the following section, some of which will depend at various points on similar arguments.

Let $(M, \underline{\omega})$ be a compact 4-manifold with a hypersymplectic structure and let $\phi$ be the associated closed $G_2$-structure on $X = M \times \mathbb{T}^3$ as defined in \eqref{toric-G2}. The first simple observation is that the $G_2$-Laplacian flow starting at $\phi$ descends to a flow of hypersympectic structures.

\begin{lemma}[Lemma 2.8 of \cite{FY}]
Let $\phi(t)$ solve the $G_2$-Laplacian flow~\eqref{G2-flow} on $M \times \mathbb{T}^3$, starting at $\phi(0)$ of the form~\eqref{toric-G2} for a hypersymplectic structure $\underline{\omega}(0)$ on $M$. Then there is a path $\underline{\omega}(t)$ of cohomologous hypersymplectic structures on $M$ for which $\phi(t)$ and $\underline{\omega}(t)$ are related by~\eqref{toric-G2}.
\end{lemma}

The evolution equation for $\underline{\omega}(t)$ has a purely 4-dimensional formulation. Recall that, for each $t$, $\underline{\omega}(t)$ determines a Riemannian metric $g(t) = g_{\underline{\omega}(t)}$ on $M$ and also matrix-valued function $Q_{ij}(t) = \frac{1}{2} \langle \omega_i(t), \omega_j(t)\rangle$ given by (half of) the pointwise inner-products of the components of $\underline{\omega}(t)$ with respect to $g(t)$. The matrix $Q(t)$ is positive definite (because $\underline{\omega}(t)$ is hypersymplectic and its components are self-dual) and hence invertible. The $G_2$-Laplacian flow for $\phi(t)$ becomes the following evolution equation for $\underline{\omega}(t)$:
\begin{equation}\label{hypersymplectic-flow}
\frac{\del \underline{\omega}}{\del t}
=
\diff \left( Q\, \diff^* \left( Q^{-1}\underline{\omega} \right) \right)
\end{equation}
More explicitly, writing $Q^{ij}$ for the components of the inverse matrix $Q^{-1}$ and with the summation convention we have the following evolution equation for each component of $\underline{\omega}$.
\begin{equation}
\frac{\del \omega_i}{\del t}
=
\diff \left( 
Q_{ik} \,\diff^* \left( Q^{kl} \omega_l \right) 
\right)
\end{equation}

To give the evolution equations for the 4-dimensional metric $g(t)$ and the matrix-valued function $Q(t)$, we first need a short digression. Write $\mathcal{P}$ for the space of symmetric positive-definite 3-by-3 matrices. By definition, $Q \colon M \to \mathcal{P}$. Since $\mathcal{P}$ is an open set in the vector space $S^2\R^3$ of all symmetric 3-by-3 matrices, we can identify the tangent space $T_P\mathcal{P}$ at any point with the vector space $S^2\R^3$ in a natural way. With this in hand, we can define a Riemannian metric on $\mathcal{P}$, which makes it into a complete symmetric space of non-positive curvature. Given $A,B \in S^2\R^3 \cong T_P\mathcal{P}$ their Riemannian inner-product is given by
\begin{equation}
\left\langle A,B \right\rangle_P = \tr (P^{-1}AP^{-1}B)
\end{equation}
The action of $\GL_+(3,\R)$ on $\mathcal{P}$ given by $G \cdot P = G P G^T$ is isometric, giving 
\[
\mathcal{P} \cong \GL_+(3,\R)/\SO(3)
\]

We will treat $Q \colon M \to \mathcal{P}$ as a map between two Riemannian manifolds, using the complete symmetric metric on $\mathcal{P}$. We write $\hat{\Delta}Q$ for the harmonic map Laplacian of $Q$. The ``hat'' is to distinguish it from the ordinary Laplacian applied to each component. The two are related by
\begin{equation}
\left( \hat{\Delta}Q \right)_{ij}
	=
		\Delta(Q_{ij}) 
		-
		\langle\diff Q_{ik}Q^{km},\diff Q_{mj}\rangle_{g_{\underline{\omega}}}
\end{equation}
Another piece of notation we will use is $|\diff Q|^2_Q$ for the norm-squared of $\diff Q \colon T_xM \to T_{Q(x)}\mathcal{P}$. Explicitly,
\begin{equation}
|\diff Q|^2_Q
	=
		\langle
		Q^{ij}\diff Q_{jk},Q^{km}\diff Q_{mi}
		\rangle_{g_{\underline{\omega}}}
\end{equation}
More generally, if $A,B$ are tensors of the same type, with values in $Q^*T\mathcal{P}$, we write $\left\langle A,B \right\rangle_Q$ for their inner-product defined using $g_{\underline{\omega}}$ and the symmetric metric on $\mathcal{P}$. We can also use the metric on $\mathcal{P}$ only on the $Q^*T\mathcal{P}$ factors. For example, we write $\langle \diff Q \otimes \diff Q\rangle_{Q}$ for the symmetric 2-tensor given by
\begin{equation}
\langle \diff Q \otimes \diff Q\rangle_{Q}(u,v)
	=
		\left\langle \nabla_u Q,\nabla_v Q \right\rangle_Q
\end{equation}

The last piece of information we need is the triple of 1-forms $\underline{\tau} = Q\, \diff^*(Q^{-1} \underline{\omega})$ on $M$. The point here is that the hypersymplectic flow is $\del_t \underline{\omega} = \diff \underline{\tau}$ whilst the torsion of the $G_2$-structure $\phi$ is given by
\begin{equation}
\TT 
    =
       - \frac{1}{2} \diff^* \phi
	=  
		- \frac{1}{2}\left( 
		\diff t^1 \wedge \tau_1 
		+
		\diff t^2 \wedge \tau_2
		+
		\diff t^3 \wedge \tau_3
		\right)
\end{equation}
Starting with $\underline{\tau}$ we write $\left\langle \underline{\tau},\underline{\tau} \right\rangle$ for the 3-by-3 matrix whose $(i,j)$-element is $\left\langle \tau_i, \tau_j \right\rangle$. We also write $\underline{\tau} \otimes \underline{\tau}$ for the 3-by-3 matrix of 2-tensors whose $(i,j)$-element is $\tau_i \otimes \tau_j$. 

With all of this in hand, we can now write the evolution equations for $g(t)$ and $Q(t)$ under the hypersymplectic flow:
\begin{align}
\frac{\del Q}{\del t}
	&=
		\hat{\Delta} Q 
		+ 
		\left\langle \underline{\tau}, \underline{\tau} \right\rangle
		-
		\frac{1}{3}\tr \left(  
		Q^{-1} \left\langle \underline{\tau}, \underline{\tau} \right\rangle
		\right)
		Q
		\label{Q-evolution}\\
\frac{\del g}{\del t}
	&=
		- 2\Ric(g)
		+ 
		\frac{1}{2} \left\langle \diff Q \otimes \diff Q  \right\rangle_Q
		+
		\tr \left( Q^{-1}\underline{\tau} \otimes \underline{\tau} \right)
		-
		\frac{2}{3} \tr \left(  
		Q^{-1} \left\langle \underline{\tau}, \underline{\tau} \right\rangle
		\right) g
		\label{metric-evolution}
\end{align}
It is interesting to compare this to the Ricci flow coupled to the harmonic map flow, introduced by Buzano (né Müller) in \cite{Muller}. If one drops all the terms involving $\underline{\tau}$, then one obtains exactly this coupled flow. The terms involving $\underline{\tau}$ are lower order, just as the metric evolution of the $G_2$-Laplacian flow is a low order adjustment to the Ricci flow. It is easily seen that $|\TT|^2=\tr \left(  
Q^{-1} \left\langle \underline{\tau}, \underline{\tau} \right\rangle
\right)$. 

The two following evolution inequalities are crucial to what follows. 
\begin{proposition}
Under the hypersymplectic flow,
\begin{align}
\left( \frac{\del}{\del t} - \Delta \right) \tr Q
	& \leqslant
		\frac{5}{3} |\TT|^2 \tr Q
		\label{heat-tr}\\
\left( \frac{\del}{\del t} - \Delta \right) |\diff Q|^2_Q
	& \leqslant
		-|\hat{\nabla}\diff Q|^2_Q
		-\frac{1}{16}|\diff Q|^4_Q
		+
		C \left( \tr Q \right)^{19}|\TT|^2|\diff Q|^2_Q
		\label{heat-dQ}
\end{align}
where $\hat{\nabla}\diff Q$ is the Hessian of $Q \colon M \to \mathcal{P}$ and $C$ is a universal constant.
\end{proposition}
Inequality~\eqref{heat-tr} follows directly from taking the trace of~\eqref{Q-evolution}. The proof of~\eqref{heat-dQ} uses ideas of Eells--Sampson in their seminal work on the harmonic map heat flow \cite{ES}. The key is the Bochner formula:
\begin{equation}
\frac{1}{2}\Delta |\diff Q|^2_Q
=
|\hat{\nabla}\diff Q|^2_Q 
+ 
g^{\alpha \beta}\left\langle \hat{\nabla}_\alpha\hat{\Delta} Q, \hat{\nabla}_\beta Q \right\rangle_Q
+
\Ric^{\alpha \beta}\left\langle \nabla_\alpha Q, \nabla_\beta Q \right\rangle_Q
-
K_{\mathcal P}
\label{bochner}
\end{equation}
Here $K_{\mathcal{P}}$ is a term involving the curvature of $\mathcal{P}$. Since this is non-positive, this term can safely be ignored. This is the same reason why harmonic map heat flow is so successful when the target has non-positive curvature, as is exploited in \cite{ES}. The second point to note is that when we include $\del_t|\diff Q|^2_Q$, the term in~\eqref{bochner} involving the Ricci curvature is cancelled by the Ricci term in the time derivative of the metric. This is exactly what occurs in Buzano's study of the harmonic map flow coupled to the Ricci flow \cite{Muller}. For the remaining parts of the proof of~\eqref{heat-dQ} we refer to the original article \cite{FY}.

\subsection{Extension assuming bounded scalar curvature}\label{survey-FY}

The main result of \cite{FY} is an extension theorem for the hypersymplectic flow. In the following, $\RR(\phi(t))$ denotes the scalar curvature of the closed $G_2$-structure $\phi(t)$ on $M \times \mathbb{T}^3$.

\begin{theorem}[Theorem~1.3 of \cite{FY}]\label{FY}
Let $\underline{\omega}(t)$ be a solution to the hypersymplectic flow~\eqref{hypersymplectic-flow} defined on a time interval $[0,s)$ with $s< \infty$. If $|\RR(\phi(t))|$ is bounded on $[0,s)$ then the flow extends beyond $t=s$ to the time interval $[0, s+ \epsilon)$ for some $\epsilon >0$.
\end{theorem} 

It is an important open question to decide if bounded scalar curvature prevents a finite time singularity for a general $G_2$-Laplacian flow (cf.~Section 10 of \cite{LW1}).

Before giving an outline of the proof of Theorem~\ref{FY}, we make a small remark. A calculation due to Bryant \cite{Bryant} shows that for \emph{any} closed $G_2$-structure $\phi$, the scalar curvature is given by
\begin{equation}\label{G2-scalar-curvature}
\RR(\phi) = - | \TT |^2
\end{equation}
where $\TT = - \frac{1}{2} \diff^* \phi$ is the torsion of $\phi$. In particular $\RR(\phi) \leqslant 0$ and so the two-sided bound in Theorem~\ref{FY} is equivalent to a lower bound on $\RR$. 

We now sketch the proof of Theorem~\ref{FY}. It is also interesting to compare along the way with what is known for Ricci flow and a general $G_2$-Laplacian flow. Suppose that the flow becomes singular at $t=s$; one can zoom in on the singularity via parabolic rescaling and analyse what happens. The first step is to prove that one can take a limit, giving a complete Riemannian manifold modelling the singularity formation. This can be done for any Ricci flow or any $G_2$-Laplacian flow with bounded scalar curvature. This is well established for Ricci flow and similar ideas apply in the $G_2$ setting as we explain below. The result is a Ricci-flat complete metric with non-zero curvature and Euclidean volume growth. What is special in the case of the hypersymplectic flow is that the limit is actually \emph{hyperk\"ahler}. This brings us to the second step. Kronheimer has classified hyperk\"ahler 4-manifolds with Euclidean volume growth  \cite{kron}. In particular they all contain a holomorphic 2-sphere (for one of their complex structures). This can then be used to deduce a contradiction. We now give some more details. 

\subsubsection*{Step 1. Parabolic rescaling and singularity models}

To describe the parabolic rescaling we begin with the Ricci flow, for which this process is well known. Suppose $(M,g(t))$ is a compact Ricci flow, defined on a maximal time interval $[0,s)$ with $s < \infty$. By Hamilton's extension theorem \cite{hamilton}, we know that $\limsup_{t \to s} \|\Rm(g(t))\|_{C^0} \to \infty$. So there exists an increasing sequence of times $t_i \to s$ such that 
\[
\| \Rm(g(t_i))\|_{C^0} = \sup_{t \leqslant t_i}\| \Rm(g(t))\|_{C^0}\quad
\text{and}\quad
\| \Rm(g(t_i))\|_{C^0} \to \infty
\]
To ease the notation we set $\Lambda_i := \| \Rm(g(t_i))\|_{C^0}$. Now let 
\begin{equation}\label{metric-parabolic-rescale}
\tilde{g}_i(t) = \Lambda_i g(\Lambda^{-1}_i t + t_i)
\end{equation}
This is a sequence of Ricci flows defined on $[-\Lambda_i t_i, 0]$, with 
\[
\| \Rm(\tilde{g}_i(t))\|_{C^0} \leqslant \| \Rm(\tilde{g}(0))\|_{C^0} = 1
\]
Now Shi's estimates \cite{LW1} show that for any $A>0$, $l \in \N$ there exists $C_{l,A}$ such that
\begin{equation}\label{derivative-bounds}
\sup_{t \in [-A,0]} \| \nabla^l \Rm(\tilde{g}_i(t))\|_{C^0} \leqslant C_{l,A}
\end{equation}
So curvature of $\tilde{g}_i$ is bounded with all derivatives. To take a limit, we also need Perelman's famous non-collapsing theorem:

\begin{theorem}[Perelman's $\kappa$-non-collapsing, \cite{Perelman}. See also \cite{KL} for this particular formulation]\label{perelman-non-collapsing}
Let $g(t)$ be a Ricci flow on a compact $n$-manifold $M$, defined for $t \in [0,s)$. There exists $\kappa>0$ such that for all $t \in [0,s)$, and for all $B_{g(t)}(x,r)$ with $r < \sqrt{s}$, 
\[
\text{if}\quad
\sup_{B(x,r)} R \leqslant \frac{1}{r^2}\quad
\text{then}\quad
\Vol_{g(t)}B_{g(t)}(x,r) \geqslant\kappa r^n
\]
\end{theorem}
With this in hand, standard convergence theorems from Riemannian geometry \cite{hamilton} enable us to take a pointed limit (in the sense of Cheeger--Gromov \cite{P}), centred at the points of maximum curvature, $\tilde{g}(t) = \lim_{i \to \infty} \tilde{g}_i(t)$. The limit is again a solution to Ricci flow, defined on an open manifold $N$ for $t \in (-\infty,0]$.

If we assume in addition that the original Ricci flow satisfies $|R(g(t))| < K$, then the rescaled limit has $R(\tilde{g}(t))= 0$ for all $t$. The evolution equation for scalar curvature under Ricci flow reads
\begin{equation}
\frac{\del R}{\del t} = -  \Delta R + 2|\Ric|^2
\end{equation}
So a scalar-flat Ricci flow is necessarily Ricci flat, and $\tilde{g}(t) = \tilde{g}$ is actually independent of $t$. We remark that, by choice of the rescaling factors and the centres in the pointed limit, $\| \Rm(\tilde{g})\|_{C^0} = 1$ and so $(N,\tilde{g})$ is not simply flat. 

The bound $|R(g(t))|<K$ also means that Perelman's non-collapsing result gives information about volume of \emph{all} balls below a certain radius: there exists $\kappa>0$ and $\rho >0$ such that for any $x \in M$, $t \in [0,s)$ and $r \leqslant \rho$, 
\begin{equation}\label{non-collapse}
\Vol_{g(t)}B_{g(t)}(x,r) \geqslant\kappa r^n
\end{equation}
After rescaling, this implies that $(N, \tilde{g})$ has Euclidean volume growth. 

Up to this stage the argument holds in all dimensions. Thanks to a theorem of M.~Simon \cite{Simon}, in dimension~4 we can make one further deduction about $(N, \tilde{g})$.
\begin{theorem}[Simon \cite{Simon}]
\label{simon-L2-curvature}
Let $g(t)$ be a solution to Ricci flow on a compact 4-manifold $M$, for $t \in [0,s)$ with $s <\infty$. Suppose moreover that $|R(g(t))| < K$ for all $t \in [0,s)$. Then there is a constant $C$ such that for all $t \in [0,s)$,
\[
\int_M |\Rm(g(t))|^2\, \dvol_{g(t)} \leqslant C
\]
\end{theorem}
Since the $L^2$-norm of curvature is scale invariant, this information passes to the limit in our parabolic rescaling and we can conclude that 
\[
\int_N |\Rm (\tilde{g})|^2\, \dvol_{\tilde{g}} \leqslant C
\]
We can summarise this rescaling argument in the following result (which is well known to experts in Ricci flow).
\begin{theorem}[Corollary 1.9 of \cite{BZ}]
Let $g(t)$ be a solution to the Ricci flow on a compact manifold $M$. Suppose that the flow exists on a maximal time interval $[0,s)$ with $s< \infty$ and, moreover, that $|R(g(t))| <K$ for all $t \in [0,s)$. Then parabolic rescaling~\eqref{metric-parabolic-rescale} produces a limit $(N, \tilde{g})$ which is complete, Ricci flat with Euclidean volume growth, but which isn't simply flat. If in addition $M$ is four-dimensional, then the limit $(N,\tilde{g})$ has finite energy, i.e.\ the $L^2$-norm of the curvature of $\tilde{g}$ is finite. 
\end{theorem}



Next consider be a $G_2$-Laplacian flow $\phi(t)$ of closed $G_2$-structures on a compact 7-manifold $X$, defined on a maximal time interval $[0,s)$ with $s< \infty$. (For the moment we work in full generality, not assuming the $G_2$-structure is determined by a hypersymplectic triple.) By Lotay--Wei's extension theorem (stated as Theorem~\ref{LW-extension} above) we know that $\limsup_{t \to s} \Lambda(\phi(t)) = \infty$, where 
\[
\Lambda(\phi) = \sup_{X} \left( |\RRm(g_{\phi(t)})|^2 + |\bm\nabla \TT(\phi(t))|^2\right)^{1/2}
\]
So there is a sequence $t_i \to s$ for which 
\[
\Lambda(\phi(t_i))
	= \sup_{t \leqslant t_i}\Lambda(\phi(t))\quad
\text{and}\quad
\Lambda(\phi(t_i)) \to \infty
\] 
We write $\Lambda_i = \Lambda(\phi(t_i))$ and make the parabolic rescaling 
\begin{equation}
\label{G2-parabolic-rescale}
\tilde{\phi}_i(t) = \Lambda_i \phi(\Lambda_i^{-1} t + t_i)
\end{equation}
to produce a sequence of $G_2$-Laplacian flows defined on $[-\Lambda_it_i, 0]$, with $\Lambda(\tilde{\phi}(t)) \leqslant\Lambda(\tilde{\phi}(0)) = 1$. (Note the associated metrics are also related by the same parabolic rescaling~\eqref{metric-parabolic-rescale} as before.)

In \cite{LW1} Lotay--Wei also prove Shi-type estimates for the $G_2$-Laplacian flow. 
From this it follows that for any $A>0$ and $l \in \N$ there is a constant $C_{A,l}$ such that 
\begin{equation}\label{G2-shi}
\sup_{X \times [-A,0]}\left(
|\bm\nabla^l\RRm(g_{\tilde{\phi}_i(t)})|^2 + |\bm\nabla^{l+1}\TT(\tilde{\phi}_i(t))|^2
\right)^{1/2}
< C_{A,l}
\end{equation}
In particular, all derivatives of curvature are bounded. We would like to take a limit, but now we come to an important difference between Ricci flow and $G_2$-Laplacian flow: there is, in general, no known analogue of Perelman's non-collapsing theorem. When $R(\phi(t))$ is bounded, however, we can appeal to a recent generalisation of Perelman's non-collapsing theorem, proved by Chen \cite{Chen}.

\begin{theorem}[Chen \cite{Chen}]
Let $g(t)$ be a path of Riemannian metrics on a compact manifold, defined for $t \in [0,s)$ and suppose that for all $t$,
\[
\| \del_t g +2 \Ric(g(t)) \|_{C^0(g(t))} \leqslant K
\]
Then $g(t)$ is $\kappa$-non-collapsed relative to scalar curvature for all $t \in [0,s)$, i.e., the conclusion of Theorem~\ref{perelman-non-collapsing} holds.
\end{theorem}

Recalling the evolution equation~\eqref{metric-evolution-G2} for $g(t)$ under $G_2$-Laplacian flow, and the fact that $\RR(\phi(t)) = -|\TT|^2$, we see that when the scalar curvature is bounded, the $G_2$-Laplacian flow enjoys the same non-collapsing as the Ricci flow:

\begin{corollary}[Chen's non-collapsing for the $G_2$-Laplacian flow with bounded scalar curvature \cite{Chen}]
Let $\phi(t)$ be a closed $G_2$-Laplacian flow on a compact 7-manifold defined for $t \in [0,s)$ with $s<\infty$. Suppose that the scalar curvature is uniformly bounded, then there exists $\kappa>0$ and $\rho>0$ such that for all $B_{g(t)}(x,r)$ with $r < \rho$, \[\Vol_{g(t)}B_{g(t)}(x,r) \geqslant\kappa r^7\]
\end{corollary}

Together with the curvature bounds from \eqref{G2-shi}, this enables us to take a limit. This time, we do not even need to take a limit of the whole flow, merely of the closed $G_2$-structures $\tilde{\phi}_i(0)$. (In the case of Ricci flow, we needed a limiting \emph{flow} to show the limit was actually Ricci flat, in the $G_2$-case this will follow from a different argument.) The non-collapsing and  curvature bounds show that the associated metrics converge (in the sense of Cheeger--Gromov \cite{P}) to a complete scalar-flat metric $\tilde{g}$ on a 7-manifold $Y$ with Euclidean volume growth. 

We now explain how to take a limit of the $G_2$-structures themselves. By equation (2.13) in \cite{LW1}, for any $G_2$-structure $\phi$, we have $\bm\nabla \phi = \TT * \phi$, where $*$ denotes an algebraic contraction of $\TT \otimes \phi$ defined using the metric. It follows that there is a constant $C$ such that for any $G_2$-structure, $\| \phi \|_{C^{k+1}} \leqslant C\| \TT \|_{C^k}$ (where $C$ depends on $k$, and $k\geqslant 1$). So, in view of~\eqref{G2-shi}, the forms $\tilde\phi_i$ are bounded in $C^{k+1}$ and so, by passing to a subsequence, we can assume they converge to a closed $G_2$-structure $\tilde{\phi}$ on the limit $Y$ (after pulling back by the diffeomorphisms involved in the Cheeger--Gromov convergence). The metric associated to $\tilde{\phi}$  is simply $\tilde{g}$ and, since this is scalar flat, by equation~\eqref{G2-scalar-curvature} $\tilde{\phi}$ is actually torsion free.

We summarise this in the following result, which is in complete analogy with what occurs for a Ricci flow with bounded scalar curvature. 

\begin{theorem}
Let $\phi(t)$ be a path of closed $G_2$-structures on a compact 7-manifold $X$ solving the $G_2$-Laplacian flow, defined on a maximal time interval $t \in [0,s)$ with $s< \infty$. Suppose, moreover, that the scalar curvature is bounded for all $t \in [0,s)$. Then the parabolic rescaling~\eqref{G2-parabolic-rescale} converges to a limiting torsion free $G_2$-structure $(Y, \tilde{\phi})$ on a manifold $Y$. The associated metric is complete, not flat, with holonomy contained in $G_2$ (and so, in particular, is Ricci flat), with Euclidean volume growth. 
\end{theorem}

We now turn to the case of a hypersymplectic flow. Let $\underline{\omega}(t)$ be a solution to the hypersymplectic flow, defined on a compact 4-manifold $M$, on a maximal time interval $t \in [0,s)$. Suppose moreover that the 7-dimensional scalar curvature $\RR(\phi(t))$ is bounded uniformly in $t$. We let $t_i \to s$ and $\Lambda_i$ be as before, for the general discussion of the $G_2$-Laplacian flow. We then consider the parabolic rescaling
\begin{equation}
\label{hypersymplectic-parabolic-rescale}
\tilde{\underline{\omega}}_i(t) 
	=
		\Lambda_i\, \underline{\omega} (\Lambda_i^{-1}t + t_i)
\end{equation}
Under this rescaling, the \emph{four}-dimensional metric $g_{\tilde{\underline{\omega}}_i}(t)$ scales as in~\eqref{metric-parabolic-rescale}. The $G_2$-structures and seven-dimensional metrics, however, do not scale as in~\eqref{G2-parabolic-rescale}, because only the 4-manifold directions are stretched whilst the $\mathbb{T}^3$ directions are unchanged.

The singularity limit we obtain is given by the following theorem.

\begin{theorem}[\cite{FY}]\label{hypersymplectic-bubble}
Let $\underline{\omega}(t)$ be a solution to the hypersymplectic flow, defined on a compact 4-manifold $M$, on a maximal time interval $t \in [0,s)$ with $s < \infty$. Suppose moreover that the 7-dimensional scalar curvature $\RR(\phi(t))$ is bounded uniformly in $t$. Them the parabolic rescaling~\eqref{hypersymplectic-parabolic-rescale} converges to a limiting hyperk\"ahler 4-manifold $(N, \tilde{\underline{\omega}})$. The metric is complete, its curvature is non-zero, with finite $L^2$-norm, and it has Euclidean volume growth. In other words, it is a non-trivial ALE gravitational instanton.  
\end{theorem}

One approach to proving this theorem would be to prove an extension result and Shi estimates directly for the hypersymplectic flow. Instead, the argument in \cite{FY} leverages the known estimates for $G_2$-Laplacian flow. This means the discussion has to pass back and forth between the 4-manifold $M$ and the 7-manifold $M \times \mathbb{T}^3$, but there is nothing essentially new involved here in producing the limiting Riemannian manifold.  The two genuinely new ideas needed are to show firstly that the limit has finite energy and secondly to show that is hyperk\"ahler. 

The proof of finite energy is inspired by Simon's proof in the case of Ricci flow. The idea is to compute time derivatives of the integral of various curvature quantities. Each integral produces a ``good'' term with the right sign and a ``bad'' term with the wrong sign. The next integral has a good term which cancels the bad one from the previous integral until eventually one is able to close the circle. Compared with Ricci flow, the evolution equations for hypersymplectic flow are more complicated, because of the torsion terms, and so one has to be correspondingly a little more imaginative with the choice of quantities to integrate. 

To show that the limit is hyperk\"ahler one must prove that the rescaled forms $\tilde{\omega}_i$ pass to the limit. It turns out that to do this it suffices to prove that $Q(t)$ is bounded in $C^2$ uniformly in $t$. This ultimately comes down to judicious use of the evolution inequalities~\eqref{heat-tr} and~\eqref{heat-dQ}. Again, we refer the reader to the article \cite{FY} for the details.

\subsubsection*{Step 2. Ruling out ALE gravitational instanton singularity models}

We are now in a position to complete the proof of Theorem~\ref{FY}. We begin with a compact hypersymplectic flow for $t \in [0,s)$ with $s< \infty$ and for which the 7-dimensional scalar curvature is bounded. We  assume for a contradiction that the flow does not extend to $t=s$. Then Theorem~\ref{hypersymplectic-bubble} enables us to take a limit via parabolic rescaling to obtain an ALE gravitational instanton which, moreover, is not simply flat. Kronhiemer has classified these \cite{kron}. We will need a single important consequence of this classification:
\begin{theorem}[Kronheimer \cite{kron}]
Let $(N,h)$ be an ALE gravitational instanton, i.e., a complete hyperk\"ahler 4-manifold with finite energy and with Euclidean volume growth. If $h$ is not flat, then there is an embedded 2-sphere $S \subset N$ which is holomorphic for one of the hyperk\"ahler complex structures.
\end{theorem}
From here the idea for deriving a contradiction is simple. The sphere $S \subset N$ is the limit of a sequence of spheres $S_{i} \subset M$ in the compact 4-manifold. We must have 
\begin{equation}
\int_{S_i} \omega_j(t_i) \to 0
\label{symplectic-classes}
\end{equation}
as $i \to \infty$. This is because the rescaled forms $\tilde{\omega}_j(t_i)$  converge to the hyperk\"ahler structure and so the rescaled integral has a finite limit. To see why this should give a contradiction, suppose for a moment that the $S_i$ were all homologous. Then the above integral would be fixed and so vanish for $j=1,2,3$. But in the limit $S$ is holomorphic for one of the hyperk\"ahler structures and so for large $i$, $S_i$ must be symplectic for one of the hyper\emph{symplectic} structures, meaning the integral for at least one value of $j$ must always be strictly positive.   

To complete the argument one must show that the classes $[S_i]$ actually take on only finitely many different values. To see this note that in the hyperk\"ahler limit, $[S]^2=-2$ by the adjunction formula. Since a neighbourhood of $S_i$ is diffeomorphic to a neighbourhood of $S$, it follows that for all $i$, $[S_i]^2=-2$. Meanwhile, by~\eqref{symplectic-classes} the evaluation of the symplectic classes is bounded. Since they generate $b_+$, we see that the positive and  part of $[S_i]$ is bounded. Since $[S_i]^2 = -2$, it follows that the negative part is also bounded and so $[S_i]$ lies in a bounded subset $F \subset H_2(M,\Z)$, which is hence finite. We know  that $S_i$ is symplectic for, say $\omega_1$ for all large $i$. Now $\omega_1$ has a smallest strictly positive value $\eta>0$ on $F \subset H_2(M,\Z)$  and so 
\[
\int_{S_i} \omega_1(t_i) \geqslant\eta >0
\]
for all large $i$, which contradicts~\eqref{symplectic-classes}. This completes the proof of Theorem~\ref{FY}. 


\section{New results}\label{new-results}

 \subsection{Torsion free hypersymplectic structures}

In this subsection we do not assume the manifold $M$ is compact. A torsion-free closed $G_2$-structure determines a Riemannian metric with  holonomy group contained in $G_2$ and so in particular the metric is Ricci flat. By contrast, a torsion free hypersymplectic structure in dimension $4$ is not necessarily hyperk\"ahler. To give examples, we recall a construction due to Donaldson \cite[Section 3]{Don2}. 

\begin{proposition}[Donaldson \cite{Don2}]\label{Donaldson}
	Any $S^1$-invariant torsion-free hypersymplectic structure $\underline{\omega}$ is (locally) determined by a convex potential function
	$u\colon P\to \mathbb{R}$ defined on an open subset $P \subset \R^3$,  and a smooth function $S: P\to \mathbb{R}^+$ satisfying
	$$\det \left( \frac{\partial^2 u}{\partial x^i \partial x^j} \right)=1;$$
	and 
	$$U^{ij}S_{ij}=0$$
	via the formula
	\[
	\omega_i 
	= 
	\alpha\wedge \dd x^i 
	+\frac{1}{2}S U^{ij}\epsilon_{jkl}\dd x^k\wedge\dd x^l\;\;, \; i=1, 2, 3
	\]
	where $\left(U^{ij} \right) = \left( \frac{\partial^2 u}{\partial x^i \partial x^j} \right)^{-1}$ and $S_{ij}=\frac{\partial^2 S}{\partial x^i \partial x^j}$. The function $\frac{1}{\sqrt{S}}$ measures the length of the $S^1$-orbit under the associated metric $g_{\underline\omega}$. The metric $g_{\underline\omega}$ is hyperk\"ahler if and only $\left( \frac{\partial^2 u}{\partial x^i \partial x^j} \right)$ is constant. 
\end{proposition}

In the above formula, $\alpha=\dd t+ \alpha_1\dd x^1 + \alpha_2\dd x^2 + \alpha_3 \dd x^3$ satisfies $\dd \alpha = - \frac{1}{2}U^{ip}S_p \epsilon_{ijk}\dd x^j\wedge \dd x^k$. The associated Riemannian metric is 
\begin{equation}
\begin{split}
g_{\underline\omega}
& =
S^{-1}\alpha^2 + S u_{ij}\dd x^i\otimes \dd x^j\\
& =
S^{-1}\dd t^2 + S^{-1}\alpha_i \dd t\otimes \dd x^i + S^{-1}\alpha_i\dd x^i\otimes \dd t
+ (Su_{ij}+ S^{-1}\alpha_i\alpha_j)\dd x^i\otimes \dd x^j
\end{split}
\end{equation}
Corresponding to the matrix of metric in the coordinates $t, x_1, x_2, x_3$, the inverse matrix is

\[
\left( g_{\underline\omega}^{ij} \right)
=
\left(
\begin{array}{cc}
S + S^{-1}\alpha U\alpha^t & - S^{-1}\alpha U\\
- S^{-1}U\alpha^t & S^{-1}U
\end{array}
\right)
= 
\left(
\begin{array}{cc}
S+S^{-1}\alpha_iU^{ij}\alpha_j &  -S^{-1}\alpha_k U^{kj}\\
-S^{-1}\alpha_k U^{ik} & S^{-1}U^{ij}
\end{array}
\right)
\]
and the matrix $Q$ is precisely $U=\left( U^{ij} \right)$. There is a related interesting description of torsion-free $G_2$-structures with torus symmetry by Madsen and Swann \cite{MS}, generalizing the Gibbons--Hawking construction of $S^1$-invariant hyperk\"ahler metrics in dimension $4$.

\begin{corollary}[Donaldson \cite{Don2}]
	There exists a (local) torsion free hypersymplectic structure $\underline\omega$ whose associated metric $g_{\underline\omega}$ is not hyperk\"ahler. 
\end{corollary}
\begin{proof}
	Let $\underline\omega$ be a torsion free hypersymplectic structure with $S^1$ symmetry determined by the data $(u, S)$ of Proposition \ref{Donaldson}. The $3$-form $\phi=\dd t^{123}- \dd t^1 \wedge \omega_1 - \dd t^2\wedge \omega_2 - \dd t^3\wedge \omega_3$ on $U\times S^1\times \mathbb{T}^3$ is a torsion free $G_2$-structure, and therefore the associated Riemannian metric $g_\phi$ is Ricci-flat. According to the formula \cite[Lemma 3.5]{FY}, 
	\[
	R
	= \frac{1}{4}|\dd Q|_Q^2
	= \frac{1}{4} Q^{ij}\nabla_\alpha Q_{jk}Q^{kl}\nabla_\beta Q_{li} g^{\alpha\beta}
	= \frac{1}{4} u_{ij}U^{jk}_{\;\;\; ,\alpha} u_{kl} U^{li}_{\;\;\;,\beta}\frac{1}{S}U^{\alpha\beta}
	= \frac{1}{4S} U^{ai}U^{bj}U^{ck} u_{abc} u_{ijk}
	\] 
	Any solution $u$ to $\det \left( \frac{\partial^2 u}{\partial x^i \partial x^j} \right)=1$ which is not a quadratic polynomial gives a metric where
	$|\dd Q|_Q^2 \neq 0$ and thus $g_{\underline\omega}$ is not scalar-flat. 
	
	As an almost explicit example, we could take $u(x_1, x_2, x_3)=r^\frac{4}{3} w(x_1)$ where $r=\sqrt{x_2^2+x_3^2}$.  Then 
	\[
	\det \left(u_{ij} \right) 
	=\frac{16}{27}(w^2 w'' - 4ww'^2)
	\]
	By ODE theory, an even function $w$ defined on some interval $(-\delta, \delta)$ such that $\det (u_{ij})=1$ exists. This non-quadratic function is strictly convex on the domain $(-\delta, \delta)\times \{x_2=x_3=0\}$.
\end{proof}

The condition of being {torsion-free} for a hypersymplectic structure $\underline\omega$ is equivalent to the following equations (according to \cite[Lemma 3.5]{FY}):

\begin{align}
\hat\Delta Q 
&
=0\label{Q-harmonic}\\
\text{Ric }g_{\underline\omega} 
&
= \frac{1}{4}\left\langle \dd Q\otimes \dd Q \right\rangle_Q
\label{Ric-equation}
\end{align}
Therefore a torsion free hypersymplectic structure leads to a harmonic map from a 4-manifold with non-negative Ricci curvature to a 5 dimensional symmetric space of non-positive curvature. In the presence of geometrical completeness of the domain and target manifolds, one should expect a Liouville type theorem to hold, implying the map must be constant. The following theorem shows this is indeed true.

\begin{theorem}\label{torsion free-hyperkahler}
	Let $(M, \omega_1, \omega_2, \omega_3)$ be a hypersymplectic structure. Suppose $\TT=0$ and the associated metric $g_{\underline\omega}$ is complete, then $\underline\omega$ is hyperk\"ahler. 
\end{theorem}
We give the proof since it is elementary and the results in the literature are not directly applicable.  Note first that by \eqref{Ric-equation}, the Ricci curvature of $g_{\underline{\omega}}$ is non-negative. Next we recall the Bochner formula \eqref{bochner}
\begin{equation*}
\frac{1}{2} \Delta |\diff Q|^2_Q
=
|\hat{\nabla}\diff Q|^2_Q
+
g^{\alpha \beta}\left\langle 
\hat{\nabla}_\alpha \hat{\Delta}Q, \nabla_\beta Q
\right\rangle_Q
+
R^{\alpha \beta} \left\langle 
\nabla_\alpha Q,\nabla_\beta Q 
\right\rangle_Q
-
K_{\mathcal{P}}
\end{equation*}
where $K_{\mathcal{P}}$ is a term involving the sectional curvature of $\mathcal{P}$:
\[
K_{\mathcal{P}} = \Rm_{\mathcal{P}}(
\nabla_\alpha Q, \nabla_\beta Q, \nabla^\beta Q, \nabla^\alpha Q)
\]
and is non positive. From this, \eqref{Q-harmonic}, \eqref{Ric-equation} and the fact that $|\nabla |\diff Q|_Q| \leqslant |\hat{\nabla}\diff Q|_Q$ (Kato's inequality) we get
\begin{equation}\label{Bochner-inequality}
\begin{split}
\frac{1}{2} \left|\dd Q\right|_Q^2 \Delta |\dd Q|_Q^2 
&=
| \hat\nabla \dd Q |_Q^2 |\dd Q|_Q^2 
+  \frac{1}{4} \left|\left\langle \dd Q\otimes \dd Q \right\rangle_Q \right|^2 |\dd Q|_Q^2
- K_\mathcal{P} |\dd Q|_Q^2\\
& \geqslant
\frac{1}{4} \left| \nabla \left| \dd Q\right|^2_Q \right|^2	
+ \frac{1}{4} \left| \frac{1}{4} |\dd Q|_Q^2 g_{\underline\omega} \right|^2 \\
& =
\frac{1}{4} \left| \nabla |\dd Q|^2_Q\right|^2 + \frac{1}{16} |\dd Q|_Q^6
\end{split}	
\end{equation}

We now need the following lemma:

\begin{lemma}\label{Calabi}
	Let $(M, g)$ be a complete Riemannian $4$-manifold with nonnegative Ricci curvature. The only nonnegative function $f$ solving the following inequality is $f \equiv 0$.
\begin{equation}\label{calabi-ineq}	
f\Delta f
	\geqslant
		\frac{1}{2} |\nabla f|^2 + \frac{1}{2} f^3
\end{equation}
\end{lemma}
Results of this kind date back to Haviland, Osserman and Calabi \cite{Ca}.  
\begin{proof}[Proof of Lemma \ref{Calabi}]
	We first claim that $\Delta f \geqslant  0$ on $M$. To see this suppose that at some point $p$,  $\Delta f < 0$. The inequality~\eqref{calabi-ineq} shows that $f(p)=0$. But this means that $p$ is a minimum of $f$ and so at $p$ we must have $\Delta f \geq0$, a contradiction.
	
	Next, for any $\epsilon>0$, define $h_\epsilon=(f+\epsilon)^\frac{1}{2}$, then $h_\epsilon$ is a strictly positive smooth function, and 
	\begin{equation}\label{barrier}
	\begin{split}
	\Delta h_\epsilon
	& =\frac{1}{2}(f+\epsilon)^{-\frac{1}{2}} \Delta f 
	- \frac{1}{4} (f+\epsilon)^{-\frac{3}{2}} |\nabla f|^2 \\
	& \geqslant\frac{1}{2} (f+\epsilon)^{-\frac{3}{2}} \epsilon \Delta f
	+ \frac{1}{4} f^3 (f+\epsilon)^{-\frac{3}{2}}\\
	& \geqslant\frac{1}{4} f^3 (f+\epsilon)^{-\frac{3}{2}}\\
	& = \frac{1}{4}( h_\epsilon^3 - 3 h_\epsilon\epsilon + \frac{3}{h_\epsilon}\epsilon^2 - \frac{\epsilon^3}{h_\epsilon^3})\\
	& \geqslant\frac{1}{4}h_\epsilon^3 - \frac{3}{4}h_\epsilon\epsilon - \frac{1}{4} \epsilon^\frac{3}{2}
	\end{split}
	\end{equation}
For any $a>0$, one checks that $v_a(x)=\frac{32a}{32-a^2 x^2}, x\in [0, \frac{4\sqrt{2}}{a})$ is the unique one-variable real-valued function satisfying 
	\[
	\left\{ \begin{array}{l}
	v''(x)+\frac{3}{x} v'(x) = \frac{1}{4} v(x)^3 \\
	v(0)=a, \; v'(0)=0
	\end{array}   \right.
	\]      
Now let $p_0\in M$ be any given point, $r_{p_0}(\cdot)=d(p_0, \cdot)$ be the distance function,  and define a local comparison function $V_a(\cdot)=v_a\left( r_{p_0}(\cdot)\right)$ on the metric ball $B_{g_{\underline\omega}}(p_0, \frac{4\sqrt{2}}{a})\subset M$. Then on this metric ball we have
	\begin{equation}\label{superlinear}
	\Delta V_a
	=v''_a+v'_a\Delta r_{p_0}\leqslant v_a''+\frac{3}{r_{p_0}}v_a'
	=\frac{1}{4}V_a^3
	\end{equation}
The function $h_\epsilon -V_a$ approaches $-\infty$ near $\partial B_{g_{\underline\omega}}(p_0, \frac{4\sqrt{2}}{a})$, therefore it takes its 
	maximum at some point $q\in B_{g_{\underline\omega}}(p_0, \frac{4\sqrt{2}}{a})$. By \eqref{barrier} and \eqref{superlinear}, at $q$, 
	
	\begin{equation*}
	0
	\geqslant
	\Delta (h_\epsilon - V_a)|_q
	\geqslant\frac{1}{4} \left( h_\epsilon(q) - V_a(q) \right)
	\left( h_\epsilon(q)^2 + h_\epsilon(q) V_a(q)+V_a(q)^2 \right)
	- \frac{3}{4}h_\epsilon(q) \epsilon - \frac{1}{4} \epsilon^\frac{3}{2}
	\end{equation*}
	which implies that
	\begin{equation*}
	\begin{split}
	h_\epsilon(p_0) - V_a(p_0)
	& \leqslant
	h_\epsilon(q) - V_a(q)\\
	& \leqslant
	\frac{3\epsilon}{h_\epsilon(q) + \frac{V_a(q)^2}{h_\epsilon(q)} + V_a(q)} 
	+ \frac{\epsilon^\frac{3}{2}}{h_\epsilon(q)^2 + h_\epsilon(q) V_a(q) + V_a(q)^2}\\
	& \leqslant
	\frac{\epsilon}{V_a(q)} + \frac{\epsilon^\frac{3}{2}}{V_a(q)^2}\\
	& \leqslant
	\frac{\epsilon}{a} + \frac{\epsilon^\frac{3}{2}}{a^2}
	\end{split}
	\end{equation*}
	
	Fixing $a>0$ and let $\epsilon\to 0$, we get $\sqrt{f(p_0)} \leqslant a$ and then letting $a\to 0$ we conclude that $f(p_0)=0$. It follows $f\equiv 0$ since $p_0$ is arbitrary. 
\end{proof}

\begin{proof}[Proof of Theorem \ref{torsion free-hyperkahler}]
	Let $f=\frac{1}{4}|\dd Q|_Q^2$ be the scalar curvature of the Riemannian metric associated to a torsion free hypersymplectic structure. It follows from \eqref{Ric-equation} and \eqref{Bochner-inequality} that $g_{\underline{\omega}}$ and $f$ satisfy the assumption of Lemma \ref{Calabi} and therefore $f$ vanishes. In other words, $Q$ is constant and $g_{\underline\omega}$ is hyperk\"ahler. 
\end{proof}

\begin{remark}
	Among the ``local conditions'' of being \emph{1) hyperk\"ahler}, \emph{
		2) torsion free} and \emph{3) Ricci flat} for a hypersymplectic structure, we know the implications \emph{1)$\implies$ 2)}, \emph{1)$\implies$ 3)} hold, and the implications \emph{2)$\implies$ 1)} and \emph{2)$\implies$ 3)} do not hold. The other implications are not clear, but we do not expect them to hold. 
\end{remark}

We note in passing that the comparison argument (applied to the metric ball $B_{g_{\underline\omega}}(p_0, A)$ instead of metric balls of larger and larger radius) in Theorem \ref{Calabi} shows the following local estimate for torsion-free hypersymplectic structures which might be of future use.

\begin{proposition}
	Let $\underline\omega=(\omega_1, \omega_2, \omega_3)$ be a torsion free hypersymplectic structure on an open set $U$. If $B_{g_{\underline{\omega}}}(p_0, A)\subset U$ with $\partial B_{g_{\underline\omega}}(p_0, A)\neq \emptyset$, then we have
	$$|\dd Q|_Q^2(p_0) \leqslant \frac{128}{A^2}$$
\end{proposition}

\subsection{Extension under weaker assumptions}

Turning to the hypersymplectic flow now, our first new result is a finite time extension for flows with an integral bound on $\TT$, strengthening Theorem \ref{FY}.

\begin{theorem}\label{extendible}
	Let $\{\underline\omega(t)\}$ be a flow of hypersymplectic structure existing on $[0, s)$, where $s< \infty$. Suppose 
	\[\int_{0}^s \sup_{M\times\{t\}} |\TT|^2 \dd t <\infty,
	\] then the flow extends across $t=s$. 
\end{theorem}
\begin{proof}
Let $f(t)=\sup_{M\times\{t\}} \tr Q$, then the inequality \eqref{heat-dQ} implies 
\begin{align*}
\frac{\dd}{\dd t} \log f
\leqslant 
\frac{5}{3} \sup_{M\times\{t\} } | \TT|^2 
\end{align*}
Integrating both sides on $[0, s)$ shows that $f$ is bounded on $M\times [0,s)$. Similarly, let $h(t)=\sup_{M\times\{t\}} |\dd Q|_Q^2$, then the inequality \eqref{heat-dQ} leads to 
\begin{align*}
\frac{\dd}{\dd t} \log h
\leqslant 
Cf(t)^{19} \sup_{M\times\{t\} }|\TT|^2
\end{align*}
which implies $h$ is bounded on $[0,s)$. The bound $|\TT|^2 \leqslant \frac{3}{2} |\dd Q|_Q^2 $ (Lemma 3.13 of~\cite{FY}) implies $|\TT|^2$ is uniformly bounded on $[0, s)$. The stated result is now a consequence of Theorem \ref{FY}. 
\end{proof}

Note that the bound $|\TT|^2 = O((s_m - t)^{-1})$ does not meet the hypotheses of Theorem~\ref{extendible}. However, there is the following interesting \emph{gap phenomenon} about this growth rate.

\begin{proposition}\label{gap}
Let $\underline{\omega}$ be a hypersymplectic flow on a compact 4-manifold defined for $t \in [0,s_m)$ where $s_m<\infty $ is maximal. Suppose moreover that $\tr Q <K$ uniformly along the flow. Then
\[
\limsup_{t\to s_m} \left[(s_m-t) \sup_{X \times \{t\}} |\TT|^2\right] \geqslant \epsilon_0
\]
for some constant $\epsilon_0>0$, depending only on $K$.
\end{proposition}
\begin{proof}
Suppose the claim in this corollary is not correct, for some choice of $\epsilon_0$ which will be specified shortly. Then there exists $s_0<s_m$ such that on $[s_0, s_m)$, 
	\[
	|\TT|^2(s_m - t)<\epsilon_0
	\]
We have the following inequality from \eqref{heat-dQ}:
\begin{align*}
(\partial_t -\Delta)\left( (s_m -t) |\dd Q|_Q^2 \right)
	& =
		- |\dd Q|_Q^2 + (s_m - t) (\partial_t - \Delta) |\dd Q|_Q^2\\
	& \leqslant 
		- | \dd Q|_Q^2 + (s_m - t) \left(- \frac{1}{16}|\dd Q|_Q^4 
		+ C\left(\tr Q \right)^{19}|\TT|^2 |\dd Q|_Q^2 \right)\\
	& \leqslant
		\left( -1+ CK^{19}|\TT|^2(s_m-t) \right) |\dd Q|_Q^2\\
	& \leqslant
		\left( -1+ CK^{19} \epsilon_0 \right) |\dd Q|_Q^2
\end{align*}
Now taking $\epsilon_0 = \frac{1}{2CK^{19}}$ we have that 
\[
(\partial_t -\Delta)\left( (s_m -t) |\dd Q|_Q^2 \right)\leqslant -\frac{1}{2} |\diff Q|^2_Q
\]
Let $h(t)=\sup_{X\times\{t\}} (s_m - t)|\dd Q|_Q^2$, then the above inequality implies 
\[
h'(t)
	\leqslant
		-\frac{1}{2(s_m-t)} h(t) 
\]
This implies 
\[
|\dd Q|_Q^2 (t)
	\leqslant 
		|\dd Q|_Q^2 (s_0) \frac{(s_m - s_0)^\frac{1}{2}}{(s_m - t)^\frac{1}{2}}
\]
From here, Theorem \ref{extendible} implies the flow extends across $s_m$, contradicting the maximality of $s_m$. 
\end{proof}

\begin{remark}
It is interesting to compare Theorem~\ref{extendible} and Proposition~\ref{gap} with what is known for the Ricci flow, where analogous results have been proved with respect to \emph{Ricci} curvature instead of scalar curvature, by Wang \cite{W1, W2}. 

There are also similarities with results in mean curvature flow. The article \cite{li-wang} shows extension of the mean curvature flow of a surface in $\R^3$ provided the mean curvature remains bounded. Meanwhile, \cite{li-wang2} shows an inequality in this setting which is analogous to Proposition~\ref{gap}, with the mean curvature there again playing the role of the torsion here.  

This suggests that it might be worth comparing closely mean curvature flow and $G_2$-Laplacian flow, with the roles of mean curvature and torsion being analogous. On this subject, it is worth pointing out that a concrete link between \emph{spacelike} mean curvature flow and the $G_2$-Laplacian flow is explored in \cite{LL}. In the setting studied there, torsion and mean curvature can be directly identified (remark~1.9 of \cite{LL}.) 
\end{remark}

\subsection{Long-time existence given an initial $C^0$-bound}

The goal of this section is to prove the following result:
\begin{theorem}\label{C0-long-time-existence}
There exists $\epsilon_0>0$ such that if the initial hypersymplectic structure  $\underline{\omega}(0)$ satisfies 
\[
\tr Q(0) < 3 + \epsilon_0
\] 
then the hypersymplectic flow $\underline{\omega}(t)$ starting at $\underline{\omega}(0)$ exists for all time. Moreover, $\tr  Q(t) < 3 +\epsilon_0$ and $|\diff Q|^2_Q(t)<C$ for all $t$. 
\end{theorem}

Note that, since $\det Q = 1$, we have $\tr Q \geqslant 3$ for any hypersymplectic structure, with equality if and only if $Q$ is the identity matrix. The hypothesis of Theorem~\ref{C0-long-time-existence} says that the initial symplectic forms $\omega_i(0)$ are $C^0$-close to being point-wise orthogonal.

The first step in the proof of Theorem~\ref{C0-long-time-existence} is to show that the bound on $\tr Q$ is preserved along the flow.

\begin{lemma}
Let $\underline{\omega}(t)$ be a hypersymplectic flow on a compact 4-manifold. If $\tr Q(0) < 2^{5/3}$ then $\tr Q(t) < 2^{5/3}$ at all later times.
\end{lemma}

\begin{proof}
Diagonalising $Q=(\lambda_1, \lambda_2, \lambda_3)$ gives
\begin{equation}\label{max-principle}
\begin{split}
(\partial_t - \Delta) \tr Q
	& = 
		-\sum_{i, j=1}^3 \lambda_j^{-1}|\dd Q_{ij}|^2 
		+ \sum_{i=1}^3 |\tau_i|^2 
		-\frac{1}{3}\sum_{i=1}^3 \lambda_i^{-1}|\tau_i|^2 \tr Q\\
	&\leqslant 
		\sum_{i=1}^3 (\frac{2}{3} -\frac{\tr Q}{3\lambda_i} ) |\tau_i|^2\\
\end{split}
\end{equation}
When $Q = I$, 
the coefficients on the right-hand side are strictly negative. This means that there is a neighbourhood of $I$ for which the right-hand side remains negative and for which we can apply the maximum principle. We now measure this neighbourhood in terms of $\tr Q$. To get an explicit expression for its size, we make the following observation:
for any $A\in [3, 2^{5/3})$ and $x, y, z>0$ satisfying $x+y+z=A,\; xyz=1$, we have
\begin{align*}
	x+y & >z\\
	y+z & >x\\
	z+x & >y
\end{align*}
The argument for this is as follows: assume (without loss of generality) that $x\leqslant y\leqslant z$, then it suffices to show the inequality
\begin{equation}
x+y- (A- (x+y))>0
\label{trace-ineq}\end{equation}
under the assumption that $x + y + x^{-1}y^{-1}=A$. Fix $x$, then smaller $y$ gives bigger $z$, and thus smaller $x+y - z$. Thus, it suffices to consider the case $y=x$. This critical value $x_0$ must solve 
\begin{equation}
2x_0 + x_0^{-2}=A
\label{x0-critical}
\end{equation}
(where $x_0$ is the root of this equation on $(0,1)$). Meanwhile, the critical value of $A$ occurs when~\eqref{trace-ineq} is saturated, i.e., when $x_0 = A/4$. Substituting into~\eqref{x0-critical}, we see that the critical value of $A$ is $A= 2^{5/3}$. It is simple to check that for any $A < 2^{5/3}$ the solution $x_0$ of~\eqref{x0-critical} satisfies $x_0 > \frac{A}{4}$  and so for any $A<2^\frac{5}{3}$~\eqref{trace-ineq} holds.

Now let $Q\in \mathcal{P}$ with $\det Q=1$ and $\tr Q<2^{5/3}$. By the previous paragraph, the three eigenvalues $\lambda_1, \lambda_2, \lambda_3$ of $Q$ satisfy $2\lambda_i - \tr Q <0$ which makes the right-hand side of the evolution equation for $\tr Q$ non-positive. The maximum principle implies that  $\tr Q$ is non-increasing, and so the bound $\tr Q<2^{5/3}$ holds for all time. 
\end{proof}

Given $\epsilon >0$ we define $\delta(\epsilon)>0$ to be the infimum of all $\delta>0$ such that $\tr Q < 3 + \epsilon$ implies $\min\{\lambda_1, \lambda_2, \lambda_3\} \geqslant 1 - \delta$ and $\max\{\lambda_1, \lambda_2, \lambda_3\} \leqslant 1 + \delta$. 

\begin{lemma}
Suppose that $\tr Q(t) < 3 +\epsilon$. Then for any choice of $\eta \in [0,1]$, we have
\[
\left(\del_t - \Delta\right) \tr Q
	\leqslant 
		-\eta (1 - \delta(\epsilon))|\diff Q|^2_Q
		- \frac{1}{3} \big(
					\tr Q - (2+ \eta)(1+\delta(\epsilon))
					\big) |\TT|^2
\]
\end{lemma}
\begin{proof}
Consider the first term in the heat inequality~\eqref{max-principle} for $\tr Q$. In the proof of Lemma~3.13 of~\cite{FY}, the following inequality is established for this term:
\[
\sum_{i,j} \lambda_j^{-1} |\diff Q_{ij}|^2 \geqslant \frac{1}{3} \sum |\tau_i|^2
\]
With this in hand we have
\begin{align*}
\sum_{i,j}\lambda_j^{-1} |\diff Q_{ij}|^2 
	&\geqslant
		\eta (1- \delta(\epsilon)) 
		\sum_{i,j} \lambda_i^{-1} \lambda_j^{-1} |\diff Q_{ij}|^2
		+
		\frac{1}{3}(1 - \eta)\sum_{i}|\tau_i|^2
\end{align*}
We now have
\begin{align*}
(\del_t - \Delta) \tr Q
	&=
		- \sum_{i,j} \lambda_j^{-1} |\diff Q_{ij}|^2
		+ \sum_i |\tau_i|^2
		- \frac{1}{3} \tr Q \sum_{i} \lambda_i^{-1} |\tau_i|^2\\
	&\leqslant
		-\eta(1-\delta(\epsilon)) |\diff Q|^2_Q
		+
		\sum_i\left\{\left(
		1 -\frac{1 - \eta}{3} -\frac{\tr Q}{3\lambda_i}
		\right)|\tau_i|^2\right\}
\end{align*}
Now we recall that 
\[
|\TT|^2 = \sum_{i} \lambda_i^{-1} |\tau_i|^2 \geqslant (1+ \delta(\epsilon))^{-1} \sum |\tau_i|^2
\]
From here the stated inequality follows.
\end{proof}

\begin{proposition}\label{almost-there}
There exists $\epsilon_0 >0$ such that if the initial hypersymplectic structure $\underline{\omega}(0)$ has 
\begin{equation}
\tr Q + |\diff Q|^2_Q \leqslant 3 + \epsilon_0
\label{tr-dQ-bound}
\end{equation}
then the hypersymplectic flow starting at $\underline{\omega}(0)$ exists for all time. 
\end{proposition}
\begin{proof}
Suppose that the flow exists for $t \in [0,s)$. We will show that the bound \eqref{tr-dQ-bound} holds for all $t\in [0,s)$. From this it will follow that $|\diff Q|^2_Q \leqslant \epsilon_0$ along the flow and so, since $|\TT|^2 \leqslant \frac{3}{2}|\diff Q|^2_Q$, it follows that $|\TT|^2$ is bounded along the flow and so, by Theorem~\ref{FY}, the flow extends past $t=s$. 

To prove that~\eqref{tr-dQ-bound} holds for all $t \in [0,s)$, let
\[
I = \left\{ u \in [0,s) 
    : 
    \sup_{M\times [0,u]} 
    \left(\tr Q + |\diff Q|^2_Q\right) 
      \leqslant 
        3 + \epsilon_0\right\}
\]
By continuity, $I$ is a closed subset of $[0,s)$. We will show that, for $\epsilon_0>0$ sufficiently small, it is also open and hence $I=[0,s)$ as claimed. 

Let $u\in I$, so that \eqref{tr-dQ-bound} holds on $[0,u]$. By continuity, there exists $\delta>0$, such that for all $t \in [0,u+\delta)$, 
\[
\tr Q + |\diff Q|^2_Q < 3 + 2\epsilon_0
\]
Now take $\eta = 1/2$ and $\epsilon = 2\epsilon_0$ in the previous Lemma and choose $\epsilon_0>0$ small enough that $\delta(\epsilon) \leqslant 1/10$. For $t \in [0,u+\delta)$, we have 
\[
\left( \del_t - \Delta\right) \tr Q
	\leqslant
		-\frac{9}{20} |\diff Q|^2_Q - \frac{1}{12} |\TT|^2
\]
Combining this with the heat-inequality~\eqref{heat-dQ} for $|\diff Q|^2_Q$ we see that for all $t \in [0,u+\delta)$, 
\begin{equation}
\left( \del_t - \Delta \right) \left( \tr Q + |\diff Q|^2_Q\right)
\leqslant
\left( C (\tr Q)^{19} |\diff Q|^2_Q - \frac{1}{12} \right) |\TT|^2
-
\frac{9}{20} |\diff Q|^2_Q
\label{heat-trQ-dQ}
\end{equation}
Now we take $\epsilon_0>0$ small enough to also ensure that $C4^{19}2\epsilon_0 < \frac{1}{12}$. With this choice of $\epsilon_0$, we see that for all $t \in [0,u+ \delta)$ the right-hand side of~\eqref{heat-trQ-dQ} is non-positive. It follows that the supremum of $\tr Q + |\diff Q|^2_Q$ is non-increasing on $[0,u+\delta)$ and so~\eqref{tr-dQ-bound}, which a priori holds on $[0,u]$, actually holds on all of $[0,u+\delta)$. This shows that $I$ is open and hence $I=[0,s)$. 
\end{proof}

\begin{proof}[Proof of Theorem~\ref{C0-long-time-existence}]
Let $\underline{\omega}_0$ be an initial hypersymplectic structure with $\tr Q < 3 + \epsilon_0$ where $\epsilon_0$ is that appearing in the hypotheses of Proposition~\ref{almost-there}. Let $K >0$ be a large constant and consider the rescaled hypersymplectic structure $\underline{\omega}_0' = K^2 \underline{\omega}_0$. This leaves $Q$ untouched, but rescales $|\diff Q|^2_Q$ by $K^{-2}$. So if $K$ is chosen large enough, the new starting hypersymplectic structure $\underline{\omega}'_0$ satisfies
\[
\tr Q + |\diff Q|^2_Q < 3 +\epsilon_0
\] 
So the hypersymplectic flow $\underline{\omega}'(t)$ starting at $\underline{\omega}'_0$ exists for all time. Then $\underline{\omega}(t) = K^{-2} \underline{\omega}'(t)$ is the sought-after global solution to hypersymplectic flow starting at $\underline{\omega}_0$. 
\end{proof}

\subsection{Convergence at infinity under assumptions}

The main goal of this subsection is to study the following natural question:

\begin{question}
	Suppose we have a compact global solution of hypersymplectic flow $\underline{\omega}(t)$ for $t \in [0,\infty)$ on a compact $4$-manifold $M$. Assume that
\begin{enumerate}
\item $\sup_{M\times [0,\infty)} \tr Q<\infty$
\item  $\sup_{M\times[0,\infty)} |\TT|^2<\infty$
\end{enumerate}
Do we get convergence as $t \to \infty$?
\end{question}
Examples of such flows are given by Theorem~\ref{C0-long-time-existence} of the previous section. We give two partial answers to this question. The first is:

\begin{theorem}[Convergence with $\Lambda$ and $Q$ bounded]\label{first-convergence}~
Let $(M^4, \underline\omega(t))$ be a compact global solution of the hypersymplectic flow \ref{hypersymplectic-flow}. Assume that
\begin{enumerate}
\item $\sup_{M\times [0,\infty)} \tr Q<\infty$
\item $\sup_{[0,\infty)} \Lambda(\phi(t)) < \infty$
\item $\chi(M) \neq 0$. 
\end{enumerate}
Then for any $t_i\to \infty$, there exists a subsequence $t_{i_k}$ such that 
\[
(M, g_{\underline\omega(t_{i_k})}, \underline\omega(t_{i_k}))\xrightarrow{C^\infty} 
(M, g_{\underline\omega_\infty}, \underline\omega_\infty)
\quad
	\text{ as }k\to \infty
\]
for some hyperk\"ahler structure $(M, g_{\underline\omega_\infty}, \underline\omega_\infty)$ on $M^4$.
\end{theorem}

And the second partial answer is

\begin{theorem}[Convergence with Ricci lower bound and diameter upper bound]\label{second-convergence}~
Let $(M^4, \underline\omega(t))$ be a compact global solution of the hypersymplectic flow \ref{hypersymplectic-flow}.  Assume that
\begin{enumerate}
\item
	$\Ric g_{\underline\omega(t)} \geqslant-C$
\item
	$ \sup_{[0, \infty)} \diam(M, g_{\underline\omega(t)})<\infty$
\item
	$\sup_{M\times [0, \infty)} \tr Q < \infty$
\item	
	$\sup_{M\times [0, \infty)} |\TT|^2<\infty$
\end{enumerate}
Then for any $t_i\to \infty$, there exists a subsequence $t_{i_k}$ such that 
\[
(M, g_{\underline\omega(t_{i_k})}, \underline\omega(t_{i_k}))\xrightarrow{C^\infty} 
(M, g_{\underline\omega_\infty}, \underline\omega_\infty)
\quad
	\text{ as }k\to \infty
\]
for some hyperk\"ahler structure $(M, g_{\underline\omega_\infty}, \underline\omega_\infty)$ on $M^4$.\end{theorem}

\subsubsection*{A priori bounds}
In order to obtain these convergence theorems, we first derive two useful a~priori estimates. 

\begin{lemma}\label{dQ-bound-at-infinity}
	Assume $\sup_{M\times [0,\infty)} \tr Q=K<\infty$ and $\sup_{M\times[0,\infty)} |\TT|^2=\frac{\beta}{2}<\infty$, then there is an explicit upper bound: 
	$$\sup_{M\times [0,\infty)} |\dd Q|_Q^2 \leqslant \max\{ \sup_M |\dd Q|_Q^2(0), 20 CK^{19}\beta\}$$
\end{lemma}

\begin{proof}
	By~\eqref{heat-dQ} (proved in \cite[Proposition 4.5]{FY}), there exists a constant $C$ such that
	
	\[
	(\del_t -\Delta) |\diff Q|^2_Q
	\leq
	-|\hat{\nabla}\diff Q|^2_Q
	-\frac{1}{16} |\diff Q|^4_Q
	+ C (\tr Q)^{19} |\mathbf{T}|^2 |\diff Q|^2_Q
	\]
	Let $f(t)=\sup_{M} |\dd Q|_Q^2(\cdot, t)$, then 
	we have 
	\begin{equation}\label{fdecreasing}
	f'\leqslant -\frac{1}{16} f^2 + \frac{C}{2} K^{19}\beta f 
	\leqslant -\frac{1}{16}(f - 4 CK^{19} \beta)^2 + C^2 K^{38} \beta^2 
	\end{equation}
Now, the maximum principle applied to (\ref{fdecreasing}) shows that 
	\begin{equation}
	\sup_{[0,\infty)}f \leqslant \max\{f(0), 8 CK^{19}\beta\}
	\end{equation}
Indeed if we suppose that $\sup_{[0,\infty)} f > \max\{ f(0), 8 CK^{19}\beta\}$, then there exists $a\in (0, \infty)$ such that $f(a)>\max\{ f(0), 8 CK^{19}\beta\}$. 
	Thus $f'(a) < 0$ by the inequality. Therefore, at the maximal point $b$ of $f$ on $ (0,a)$,  
	$$0= f'(b) \leqslant - \frac{1}{16} (f(b) - 4 CK^{19}\beta)^2 + C^2 K^{19} \beta^2 <0$$
	which gives the contradiction.
\end{proof}

\begin{lemma}[Energy bound]\label{energy-bound}
	Assume $\sup_{M\times [0,\infty)} \tr Q=K<\infty$ and $\sup_{M\times[0,\infty)} |\TT|^2=\frac{\beta}{2}<\infty$, then 
	$$\sup_{[0,\infty)} \int_M |\Rm|^2\leqslant C\left(\beta, K, |\dd Q_0|_{Q_0}^2, ||\Rm(g_0)||_{L^2}\right)$$
	and for any $\alpha>0$, 
	$$\int_\alpha^{\alpha+1}\int_{M^4} |\Ric|^4 + |\hat\Delta Q|_Q^4 + |\nabla\Ric|^2 + |\nabla\hat\Delta Q|_Q^2 + |\hat\nabla\dd Q|_Q^2+
	|\Ric|^2|\hat\nabla \dd Q|_Q^2 + |\hat\Delta Q|_Q^2|\hat\nabla\dd Q|_Q^2< G$$
	for some constant $G= G(\beta, K, |\dd Q_0|_{Q_0}, ||\Rm(g_0)||_{L^2})$.
\end{lemma}

\begin{proof}
	We begin by recalling inequality \cite[Equation (29)]{FY} which says that if the flow exists on finite time interval $[0, T_m)$ and $|\TT|^2\leqslant \frac{\beta}{2}$ on this interval, then 
	
	\begin{equation}\label{EvolvingZquantity}
	\begin{split}
	\frac{d}{dt}\int \mathcal{Z}(A_1, A_2, A_3) 
	& \leqslant \int -\frac{1}{\beta^2} \Big( |\RRic|^4 + |\RRic|^2 |\hat\nabla\dd Q|_Q^2 + |\bm\nabla\RRic|^2 + |\hat\nabla\dd Q|_Q^2\Big) \\
	& \qquad + F\Big( \mathcal{Z}(A_1, A_2, A_3) + I +1\Big)
	\end{split}
	\end{equation}
	for some constant $F= F(|\dd Q_0|_{Q_0}, \tr Q_0, \beta, T_m)$. Here 
	$$\mathcal{Z}(A_1, A_2, A_3)
	=\frac{|\RRic|^2}{\RR+\beta}+A_1 |\RRic|^2|\dd Q|_Q^2 + A_2 |\RRic|^2 + A_3 |\dd Q|_Q^2 $$ for some positive constants $A_1, A_2, A_3$ depending only on $K, \beta, |\dd Q_0|_{Q_0}$. 	

The only essential difference between the global solution and the finite time situation in \cite{FY} is that we do not have an a~priori bound of $\tr Q$ in terms of $\beta$ and $T_m$. However, we know 
	$|\dd Q|_Q^2$ is uniformly bounded in terms of $\beta, K$ and $|\dd Q_0|_{Q_0}^2$ according to Lemma \ref{dQ-bound-at-infinity}, which suffices to give the inequality above.  
	
	Using Cauchy-Schwarz inequality, we get
	\begin{equation}\label{EvolvingZquantityA}
	\begin{split}
	\frac{d}{dt}\int \mathcal{Z}(A_1, A_2, A_3)  
	& \leqslant  -\frac{1}{\beta^2\text{Vol}(g_{\underline{\omega}(t)})} \left( \int |\RRic|^2 \right)^2\\
	\\
	&\qquad\qquad
	- \frac{1}{\beta^2}\int \Big( |\RRic|^2 |\hat\nabla\dd Q|_Q^2 + |\bm\nabla\RRic|^2 + |\hat\nabla\dd Q|_Q^2\Big) \\
	& \qquad\qquad\qquad\qquad +
		 F\int \mathcal{Z}(A_1, A_2, A_3)  + F(\chi  + \text{Vol}(g_{\underline{\omega}(t)}))\\
	& \leqslant  -\frac{1}{ A \text{Vol}(g_{\underline{\omega}(t)})} \left( \int \mathcal{Z}(A_1, A_2, A_3) \right)^2\\
	&\qquad\qquad
	- \frac{1}{\beta^2}\int \Big( |\RRic|^2 |\hat\nabla\dd Q|_Q^2 + |\bm\nabla\RRic|^2 + |\hat\nabla\dd Q|_Q^2\Big) \\
	& \qquad\qquad\qquad\qquad
		+ F\int \mathcal{Z}(A_1, A_2, A_3)  + F(\chi  + \text{Vol}(g_{\underline{\omega}(t)}))\\
	\end{split}
	\end{equation}
	
Recall from \eqref{volume-bound} the absolute volume upper bound
$\label{volume-upper-bound-flow}
	\sup_{[0, \infty)}\text{Vol}(g_{\underline{\omega}(t)})\leqslant V_0
$. Denote
	$$f=\int \mathcal{Z}(A_1, A_2, A_3) $$
	then by (\ref{EvolvingZquantityA}) we have on $[0, \infty)$, 
	
	$$f'\leqslant -\frac{1}{AV_0}f^2 + F f + F(\chi + V_0)$$
	An easy maximum principle argument, similar to that used in Lemma \ref{dQ-bound-at-infinity}, shows the following explicit uniform bound:
	\begin{equation}
	\sup_{[0, \infty)} f\leqslant \max\Big\{ f(0), \frac{1}{2} FAV_0 + \sqrt{\frac{1}{4}F^2 A^2 V_0^2 + FAV_0(\chi + V_0)} \Big\}
	\end{equation}
In particular, the 7-dimensional Ricci curvature $\bf{Ric}$ is bounded in $L^2$. By Lemma~3.5 of~\cite{FY} (and using the fact that $|\diff Q|^2_Q$ is uniformly bounded) this gives an $L^2$-bound on the 4-dimensional Ricci curvature. Now, by Chern--Gauss--Bonnet, this translates into an $L^2$-bound on the whole curvature. 
	
	Meanwhile, plugging the bound on $f$ back into inequality (\ref{EvolvingZquantity}) and integrating on $[\alpha, \alpha+1]$, we get 
	\begin{equation}
	\int_\alpha^{\alpha+1} \int |\RRic|^4 + |\RRic|^2 |\hat\nabla\dd Q|_Q^2 + |\nabla\RRic|^2 + |\hat\nabla\dd Q|_Q^2 < G = G(\beta, K, |\dd Q_0|_{Q_0}^2)
	\end{equation}
	
\end{proof}

\begin{corollary}
    Under the assumption of Lemma \ref{dQ-bound-at-infinity}, there is a sequence $t_i\to \infty$ such that $(M, \underline\omega(t_i), g_{\underline\omega(t_i)})$
	has uniform bounds on the following quantities: 
	$$|| \mathrm{R} ||_{L^\infty},\quad ||\Rm||_{L^2},\quad ||\Ric||_{L^4},\quad\Vol$$
\end{corollary}

\subsubsection*{Proof of the convergence theorems}

\begin{lemma}\label{T-to-zero}
	For a global solution of the hypersymplectic flow, if $\sup_{[0, \infty)}\int_M |\TT|^{2p}\mu_{\underline\omega(t)}<\infty$ for some $p\geqslant2$, then 
	$$\int_M |\TT|^{2q} \to 0, \;\; \forall q\in [1, p)$$
	In particular, if $\sup_{M\times [0, \infty)}|\TT|^2<\infty$, then 
	$$\int_M |\TT|^{2q} \to 0, \;\; \forall q\in [1, \infty)$$
\end{lemma}
\begin{proof}
We recall the evolution equations for $|\TT|^2$ and $\mu$ (derived in \cite{LW1}):
\begin{align}
\del_t \mu 
	&= 
		\frac{2}{3} |\TT|^2 \mu\\
\del_t |\TT|^2
	&=
		\Delta |\TT|^2
		+
		4 \boldsymbol{\nabla}^b \boldsymbol{\nabla}^a
			(\TT_a^{\phantom{a}c}\TT_{cb})
		-
		2|\bf{Ric}|^2
		+
		\frac{2}{3}|\TT|^4
		-
		4 \bf{R}^{ab}\TT_a^{\phantom{a}c}\TT_{cb} \label{heat-norm-T2}
\end{align}
To ease notation, let $f(t)=\int_{M} |\TT|^2\mu_{\underline\omega(t)}$. The evolution equation for the volume form gives 
\begin{equation}
f'(t)=\int_M (\partial_t -\Delta)|\TT|^2 + \frac{2}{3}|\TT|^4
\label{f-deriv}
\end{equation}
Now applying \eqref{heat-norm-T2}, and completing the square on the terms involving $\bf{Ric}$, it follows that the uniform upper bound of $\int_M |\TT|^4$ implies a uniform upper bound 
\[
f'(t)
 \leqslant C
 \]
(Note  the term in \eqref{heat-norm-T2} involving derivatives of $\TT$ is a divergence and so vanishes upon integration.) Meanwhile, 
\[
\int_0^T f(t)\, \diff t
=
\Vol(T) - \Vol(0)
\]
and since the $\Vol(T) \leqslant V_{\max}$ is uniformly bounded \eqref{volume-bound} it follows that  $f \in L^1[0,\infty)$. Now since $f(t) \geqslant 0$ and has bounded derivative, it follows that $f(t) \to 0$ as $t\to \infty$. The stated results are then a consequence of H\"older's inequality. 
\end{proof}

\begin{proof}[Proof of Theorem \ref{first-convergence}]
	
	By the Shi-type estimates of \cite{LW1}, the condition that $\Lambda$ is uniformly bounded on $[0, \infty)$ actually implies the uniform bound for 
	\begin{equation}
	\sup_{M\times [0, \infty)} (|\bm\nabla^k\RRm|^2+|\bm\nabla^{k+1}\TT|^2)^\frac{1}{2}\leqslant C_k
	\end{equation}
	for all $k$. 
	
	Lemma \ref{T-to-zero} implies in this situation $\int_{M} |\TT|^2\mu_{\underline\omega(t)} \to 0$ as $t \to \infty$. 
	
	We now invoke the topological assumption $\chi(M)\neq 0$. We claim that it means we can find $p_t\in M$ for all $t\in [0,\infty)$ such that 
	\begin{equation}\label{somewhere-noncollapse}
	\text{inj} (M, g_{\underline\omega(t)}, p_t)\geqslant\epsilon_0
	\end{equation}
	for some positive constant $\epsilon_0$ depending only on $C$, the global upper bound of the Riemannian curvature. If not, Cheeger--Gromov's Decomposition Theorem \cite[Theorem 0.1]{CG} assures the existence of an $F$-structure of positive rank on $M$, which is in contradiction with $\chi(M)\neq 0$.

Next, we claim that 
	\begin{equation}
	\sup_{[0, \infty)}\diam(M, g_{\underline\omega(t)})<\infty
	\end{equation}
To see this we argue by contradiction. If there exists a sequence $t_i\to \infty$ with $\diam(M, g_{\underline\omega(t_i)})\to \infty$, then the condition \eqref{somewhere-noncollapse} implies	

$$({M}, {g}_{\underline\omega(t_i)}, \underline\omega({t_i}), p_i)\
	\xrightarrow{C^\infty\;\; \text{Cheeger--Gromov}}
	({M}_\infty, g_{\underline{\omega}_\infty}, \underline\omega_\infty, p_\infty)$$
(This follows similar reasoning as in Section~\ref{survey-FY} with \eqref{somewhere-noncollapse} replacing Chen's non-collapsing Theorem.) 
	
	The convergence is in the following sense. There is an exhaustion of $M_\infty =\cup_m\Omega_m$ and a sequence of diffeomorphisms $F_m: \Omega_m\to M$ such that 
	$$F_m^*(g_{\underline\omega(t_i)}, \underline{\omega}(t_i)) \xrightarrow{C^{l}} (g_{\underline{\omega}_\infty}, \underline{\omega}_\infty)$$
	on any relatively compact subest $\Omega\subset\subset M_\infty$ and for any $l\in \mathbb{N}$. 
	
	The limit hypersymplectic structure is complete and satisfies
	\begin{itemize}
		\item $\diam({M}_\infty, g_{\underline{\omega}_\infty})=\infty$;
		\item $\Vol({M}_\infty, g_{\underline{\omega}_\infty})\leqslant \lim_{i\to \infty} \int_{{M}} \mu_{\underline\omega(t_i)} < \infty$;
		\item $\int_{{M}_\infty} |\TT_\infty |^2\mu_{\underline{\omega}_\infty}\leqslant \lim_{i\to \infty}\int_{{M}}|\TT(\phi(t_i))|^2\mu_{\underline\omega(t_i)} = 0$. 
	\end{itemize}
	
	This implies the closed $G_2$-structure $\phi_\infty$ associated with the limit hypersymplectic structure $\underline\omega_\infty$ is torsion free and so, in particular, is Ricci-flat. Since $g_{\underline{\omega}_\infty}$ on $M_\infty$ is complete, the metric on $M_\infty\times \mathbb{T}^3$ of the form $g_{\phi_\infty}=g_{\underline\omega_\infty}\oplus Q_\infty$ is also complete. Moreover, $(M_\infty\times\mathbb{T}^3, g_{\phi_\infty})$ has finite volume. 
	This is a contradiction since a non-compact complete Ricci-flat $7$-manifold has infinite volume by Calabi and Yau's volume growth lower bound. (We note in passing that this contradiction could also be derived directly on the 4-manifold via Theorem \ref{torsion free-hyperkahler} stating that torsion free hypersymplectic structure on $4$-manifold whose metric is complete must be Ricci-flat.)
	
	Thus under the assumption of the bound of $\Lambda$, the diameter is uniformly bounded from above and $(M, g_{\underline\omega(t)}, \underline\omega(t))$ is uniformly noncollapsing by the Bishop--Gromov Comparison theorem. 
	
	Take any sequence $t_i\to \infty$, the compactness results outlined in Section~\ref{survey-FY} imply that there is a subsequence $t_{i_k}$, a hyperk\"ahler structure $(M,\underline{\omega}_\infty)$ and a sequence of diffeomorphisms 
$F_k: M\to M$
such that 
\[
F_k^*(\underline{\omega}(t_{i_k})) \xrightarrow{C^\infty} \underline{\omega}_\infty, \;\; \text{ as }k\to \infty \qedhere
\]
\end{proof}


\begin{proof}[Proof of Theorem \ref{second-convergence}]

	If $\diam(M, g_{\underline\omega(t)})$ is bounded above by $D$ and Ricci curvature is uniformly bounded from below, then Bishop--Gromov Comparison Theorem implies non-collapsing, i.e.
	\begin{equation}\label{global-noncollapsing}
	\Vol_{g_{\underline\omega(t)}}
	\left( B_{g_{\underline\omega(t)}}(p, r) \right)
		\geqslant
			\kappa r^4, \; \forall r \in [0, D], \; p\in M, \; t\in [0, \infty)
	\end{equation}
	for some $\kappa$ depending only on $D, C$. We will show that $\Lambda$ is uniformly bounded for the family on $[0,\infty)$. If it were not, we could pick $t_i$ to be such that $\Lambda$ achieves the maximum on $[0, t_i]$ at some point $p_i$ on the time $t_i$-slice. Then a parabolic rescaling of the flow based at $(p_i, t_i)$ by a factor $\Lambda(p_i, t_i)$ will lead to a Cheeger--Gromov limit flow of hypersymplectic structrures which is Ricci flat and thus static (Here we need the global bound of $|\TT|^2$ and $\tr Q$ to ensure the convergence on the $4$-manifold).  Lotay-Wei showed the uniform bounds of
	$(|\bm\nabla^k \RRm |^2 + |\bm\nabla^{k+1}\TT|^2)^\frac{1}{2}$ on $(-\infty, 0]$ for all $k$ for the rescaled sequence of flows, we thus have a smooth convergence of $t=0$. The limit is now hyperk\"ahler with Euclidean volume growth (by the noncollapsing condition \ref{global-noncollapsing}) and $||\Rm||_{L^2}$ finite by Lemma \ref{energy-bound}, and thus must be an ALE gravitational instanton classified by Kronheimer. We get a contradiction just as in Section~\ref{survey-FY} and the proof of the extension result in \cite[Section 6]{FY}.
	
	From here, the result follows from the previous convergence theorem.
\end{proof}

\bigskip

{ \footnotesize
	\noindent\textsc{J.~Fine},\\ D\'epartement de math\'ematiques, Universit\'e libre de Bruxelles, Brussels 1050, Belgium.\\ \textit{E-mail addresses:} \texttt{\href{mailto:joel.fine@ulb.ac.be}{joel.fine@ulb.ac.be}}
	}
\bigskip
	
{\footnotesize	
	\noindent\textsc{C.~Yao},\\ Institute of Mathematical Sciences, ShanghaiTech University, 393 Middle Huaxia Road, Pudong, 
	Shanghai, 201210 China.\\ \textit{E-mail addresses:}  \texttt{\href{mailto:yaochj@shanghaitech.edu.cn}{yaochj@shanghaitech.edu.cn}}
	}

\end{document}